\documentclass[11pt,a4paper]{article}
\usepackage[T1]{fontenc}
\usepackage{mathptmx}
\usepackage{amsmath,amssymb,amsthm,mathtools}
\usepackage[margin=27mm,headheight=14pt]{geometry}
\usepackage{booktabs,array}
\usepackage{graphicx}
\graphicspath{{figures/}{output/pdf/figures/}}
\usepackage{microtype,fancyhdr}
\usepackage{xcolor}
\definecolor{citegreen}{RGB}{0,128,0}
\usepackage[colorlinks=true,citecolor=citegreen,linkcolor=blue,urlcolor=blue]{hyperref}
\usepackage{url}
\numberwithin{equation}{section}
\newtheorem{theorem}{Theorem}
\newtheorem{lemma}{Lemma}[section]
\newtheorem{proposition}[lemma]{Proposition}

\theoremstyle{remark}
\newtheorem{remark}[lemma]{Remark}

\newcommand{\R}{\mathbb R}

\newcommand{\D}{\mathcal D}
\newcommand{\dd}{\,\mathrm d}
\newcommand{\supp}{\operatorname{supp}}
\newcommand{\dist}{\operatorname{dist}}
\newcommand{\loc}{\mathrm{loc}}

\newcommand{\eps}{\varepsilon}
\newcommand{\M}{\mathcal M}

\newcommand{\one}{\mathbf 1}
\title{\bfseries An inverse moving point source problem in inhomogeneous acoustic media: uniqueness, stability and reconstruction}
\author{ Dong Qiu\thanks{\raggedright School of Mathematical Sciences, Zhejiang University, Hangzhou 310058, China. \mbox{Email: \href{mailto:qiudong@zju.edu.cn}{\texttt{qiudong@zju.edu.cn}}}.} \and Xiang Xu\thanks{\raggedright School of Mathematical Sciences, and Center for Interdisciplinary Applied Mathematics, Zhejiang University, Hangzhou 310058, China. \mbox{Email: \href{mailto:xxu@zju.edu.cn}{\texttt{xxu@zju.edu.cn}}}.} \and Yue Zhao\thanks{\raggedright School of Mathematics and Statistics, and Key Lab NAA-MOE, Central China Normal University, Wuhan 430079, China. \mbox{Email: \href{mailto:zhaoy@ccnu.edu.cn}{\texttt{zhaoy@ccnu.edu.cn}}}.} }
\date{}

\begin{document}
\maketitle
\vspace{-1.2em}
\begin{abstract}

This paper is concerned with the simultaneous recovery of the trajectory and intensity of an unknown moving point source in an inhomogeneous acoustic medium. 
We prove that the whole-boundary measurements over a finite time interval uniquely determine the moving source. Furthermore, under suitable propagation and rank assumptions, 
Lipschitz stability of determining both the trajectory and intensity is derived using waveforms measured at four fixed receivers outside the source region. A novel Volterra--Euler reconstruction scheme is developed for solving the inverse problem, and the error estimate is derived for the proposed scheme.
Numerical experiments are presented to demonstrate the effectiveness of the proposed numerical method.
\end{abstract}

\noindent\textit{Keywords:} moving point sources, inhomogeneous medium, uniqueness, stability, reconstruction.
\medskip

\noindent\textit{2020 Mathematics Subject Classification:} Primary~35R30; Secondary~35L05.

\section{Introduction}

The inverse moving point source problem seeks to determine the trajectory and intensity of a source from acoustic measurements. Such problems arise in sound source localization and tracking, with applications to traffic noise identification \cite{Cousson2019}, aircraft localization \cite{Martin2016}, and underwater target tracking \cite{Xerri2002}. In this paper, we study the simultaneous recovery of the trajectory and intensity of a moving point source in an inhomogeneous acoustic medium.

In homogeneous acoustic media, algebraic and real-time reconstruction methods have been developed using boundary measurements~\cite{NIO2012,Ohe2020}. For a prescribed periodic signal, measurements at four receivers uniquely determine the source trajectory~\cite{HKZ2020}. For a known constant intensity, trajectory uniqueness and Lipschitz stability were established from measurements at six suitably chosen boundary points~\cite{AET2022}. \cite{WKT2023} established uniqueness and stability results for the joint recovery of a subsonic source trajectory and its intensity from boundary measurements, and developed a Bayesian reconstruction framework. 

Related developments include multiple-source reconstruction, multifrequency imaging, and electromagnetic moving source problems. Simultaneous reconstruction of multiple acoustic source trajectories and intensities was studied in~\cite{LW2026} under prescribed kinematic information and additional geometric and algebraic conditions. The proposed method combines arrival-time ODEs, linear systems for the intensities, and a direct imaging correction. Multifrequency far-field and near-field methods recover partial geometric information about the trajectory~\cite{GHM2023,MGH2024}, while direct imaging of multiple moving acoustic sources was developed in~\cite{WCL2024}. \cite{HKLZ2019} investigated inverse moving source problems in electrodynamics in homogeneous media. With the temporal source vector prescribed, trajectory reconstruction and Lipschitz stability were subsequently established from tangential magnetic fields at four noncoplanar receivers through nonlinear ODEs for the receiver--source distances~\cite{LHZ2026}.  In all the aforementioned works, the medium is required to be homogeneous and thus the explicit Green function is available. 
For acoustic waves in inhomogeneous media, a low-frequency trajectory reconstruction method was proposed in~\cite{WGL2017}, using a field approximation that reduces the influence of the background inhomogeneity. We also refer the reader to \cite{Borcea, Borcea1} on other interesting topics on inverse problems involving moving 
inhomogeneities.

Consider the acoustic wave equation
\begin{equation}
 P u:=n(x)\partial_t^2u-\Delta u=F(x,t),
 \qquad (x,t)\in\R^3\times\R,
 \label{eq:wave}
\end{equation}
supplemented by the causal condition
\begin{equation}
 u=0\qquad\text{for }t<0.
 \label{eq:causal}
\end{equation}
Throughout this paper, we assume that the background coefficient is a known real function and satisfies
\begin{equation}
 n(x)=1+q(x),\qquad q\in C_c^\infty(\R^3),
 \qquad 0<n_0\leq n(x)\leq n_1.
 \label{eq:medium}
\end{equation} 
We consider a single moving point source,
\begin{equation}
 F(x,t)=\lambda(t)\delta_{a(t)}(x),\qquad
 a:[0,T_0]\longrightarrow\R^3,
 \label{eq:source}
\end{equation}
where the intensity is extended by zero outside $(0,T_0)$.

Let $\Omega\subset\R^3$ be a bounded connected domain with smooth boundary $\Gamma=\partial\Omega$ and set $E=\R^3\setminus\overline\Omega$.
Fix a compact admissible source region $D_0\Subset\Omega$ and assume that $\supp q\Subset\Omega$. For the uniqueness problem, the trajectory--intensity pair satisfies
\begin{equation}
 a\in W^{1,\infty}([0,T_0];\R^3),\qquad
 a([0,T_0])\subset D_0,\qquad
 \lambda\in L^1(0,T_0).
 \label{eq:admissible}
\end{equation}

For the whole-boundary problem, the measured quantity is the pressure on $\Gamma$,
\begin{equation}
 \mathcal O_T(F)=u|_{\Gamma\times(0,T)}.
 \label{eq:data}
\end{equation}
For \eqref{eq:admissible}, the trace is distributional; its construction is given in Appendix~\ref{sec:direct}.

We ask whether these data determine the intensity and trajectory. Motivated by practical measurement constraints, we also investigate stability and reconstruction from waveforms at four fixed receivers outside the admissible source region. Compared with their homogeneous counterparts, inverse moving source problems in inhomogeneous media may generate a scattered tail that couples receiver measurements to the source history, posing additional challenges for the joint recovery of the trajectory and intensity.

In this paper, we prove finite-time uniqueness from whole-boundary measurements by temporal regularization and continuation outside the trajectory traces.
Furthermore, under suitable propagation and rank assumptions, we establish conditional joint Lipschitz stability of recovering the trajectory and intensity from waveforms at
only four fixed receivers outside the admissible source region. These assumptions can be realized by considering small perturbations of the homogeneous background medium. We also develop a novel Volterra--Euler reconstruction scheme and derive an error estimate to solve the inverse moving source problem.
To deal with general inhomogeneous media, we employ the spectral reconstruction method. Numerical experiments with noisy four-receiver data demonstrate the effectiveness of the proposed reconstruction methods.

Sections~\ref{sec:distribution} and~\ref{sec:four-stability} establish whole-boundary uniqueness and conditional stability from four fixed receivers, respectively. Section~\ref{sec:algorithms} develops the two reconstruction methods and demonstrates their effectiveness through four-receiver numerical experiments with time-domain and frequency-domain data. Appendices~\ref{sec:direct}, \ref{app:Green}, and~\ref{app:marching} provide the well-posedness analysis, the Green function representation, and the discrete error analysis, respectively.

\section{Inverse moving source problem: whole boundary measurements}\label{sec:distribution}
In this section, we study the unique recovery of the trajectory and intensity of a moving point source from whole-boundary pressure measurements. We work with the medium and source class specified in \eqref{eq:medium} and \eqref{eq:admissible}. The causal field and its distributional trace are constructed in Appendix~\ref{sec:direct}. To recover the entire trajectory, we additionally impose the activity condition
\begin{equation}
 \int_I |\lambda(t)|\dd t>0
 \quad\text{for every nonempty open interval }I\subset(0,T_0).
 \label{eq:activity}
\end{equation}
This condition is satisfied, in particular, if $\lambda(t)\ne0$ for almost every $t\in(0,T_0)$.

The optical metric and the travel time along a piecewise $C^1$ path $\gamma:[0,1]\to\R^3$ are
\begin{equation}
 g=n(x)\,\dd x^2,
 \qquad
 \ell_g(\gamma)=\int_0^1\sqrt{n(\gamma(s))}\,|\gamma'(s)|\dd s.
 \label{eq:metric}
\end{equation}
The minimum travel-time from $x$ to $y$ is $d_g(x,y)=\inf_{\gamma:x\to y}\ell_g(\gamma)$. Set
\begin{equation}
 r(x)=d_g(x,E),\qquad R_0=\max_{x\in D_0}r(x).
 \label{eq:radius}
\end{equation}
For $x\in\Omega$, $r(x)=d_g(x,\Gamma)$. The metric is uniformly equivalent to the Euclidean metric, so $R_0$ is finite. 

\begin{theorem}\label{thm:main}
 Let $(a_j,\lambda_j)$, $j=1,2$, satisfy \eqref{eq:admissible} and the activity condition \eqref{eq:activity}. Let $u_j$ be the causal solutions of \eqref{eq:wave} with sources $F_j=\lambda_j(t)\delta_{a_j(t)}$. Suppose that
\begin{equation}
 T\geq T_0+R_0,
 \qquad
 u_1=u_2\quad\text{on }\Gamma\times(0,T).
 \label{eq:main-data}
\end{equation}
Then
\begin{equation}
 \lambda_1(t)=\lambda_2(t), \quad a_1(t)=a_2(t)
 \quad\text{for almost every }t\in(0,T_0).
 \label{eq:atomic-equality}
\end{equation}
\end{theorem}

The proof uses three lemmas on temporal regularization, geometry of trajectory traces and one-sided unique continuation, which we state and prove below.

Choose $\rho\in C_c^\infty((-1,1))$ with $\int\rho=1$, and set
\begin{equation}
 \rho_\eps(t)=\eps^{-1}\rho(t/\eps),\qquad
 u_\eps=\rho_\eps*_t u,\qquad
 F_\eps=\rho_\eps*_t F.
 \label{eq:mollification}
\end{equation}

\begin{lemma}\label{lem:smoothing}
Let $u$ be the solution in Proposition~\ref{prop:direct}, and assume that the spatial support of $F$ is contained in a compact set $K$. For every $\eps>0$,
\begin{equation}
 Pu_\eps=F_\eps,\qquad
 \supp F_\eps\subset K\times\R,
 \label{eq:mollified-equation}
\end{equation}
and
\begin{equation}
 u_\eps\in C^\infty(\R;L^2(\R^3)),\qquad
 u_\eps\in C^\infty((\R^3\setminus K)\times\R).
 \label{eq:mollified-regularity}
\end{equation}
Moreover, $u_\eps=0$ for $t<-\eps$.
\end{lemma}

\begin{proof}
Since all coefficients of $P$ are independent of time and the convolution acts only in time, \eqref{eq:mollified-equation} follows. From \eqref{eq:negative-energy}, differentiation of the smooth convolution kernel gives $u_\eps\in C^\infty(\R;H^{-1})$, $F_\eps\in C^\infty(\R;H^{-2})$. The identity
\begin{equation}
 (1-\Delta)u_\eps
 =F_\eps+u_\eps-n\partial_t^2u_\eps
 \label{eq:bessel}
\end{equation}
has a right-hand side in $C^\infty(\R;H^{-2})$. Since $(1-\Delta)^{-1}:H^{-2}\to L^2$ is bounded, it proves the first assertion in \eqref{eq:mollified-regularity}.

On an open set disjoint from $K$, the source $F_\eps$ is zero. For each integer $j\geq0$,
\begin{equation}
 \Delta\partial_t^j u_\eps
 =n\partial_t^{j+2}u_\eps.
 \label{eq:elliptic-bootstrap}
\end{equation}
Interior elliptic regularity~\cite[Chapters~4 and 5]{Taylor2011}, applied on nested relatively compact spatial sets, improves the spatial regularity of all time derivatives at once. Starting with $L^2$, it gives $H^2_{\loc}$, then $H^4_{\loc}$, and so on. The estimates also give continuity and differentiability in time with values in each local Sobolev space. Sobolev embedding proves joint smoothness in $(x,t)$. The last assertion follows from $u=0$ for $t<0$ and $\supp\rho_\eps\subset[-\eps,\eps]$.
\end{proof}

For an open set $D\subset\R^3$, define the infima of travel times over all paths and over paths contained in $D$, respectively, by
\[
 d_g(x,y)=\inf_{\gamma:x\to y}\ell_g(\gamma),\qquad
 d_{g,D}(x,y)=\inf_{\substack{\gamma:x\to y\\\gamma\subset D}}
                         \ell_g(\gamma).
\]
The following lemma shows that removing finitely many Lipschitz curves in three dimensions does not change the infimum of travel times between points outside the curves.

\begin{lemma}\label{lem:distance}
Let $K$ be the union of finitely many compact Lipschitz curve images in $\R^3$, and let $g$ be a smooth positive definite metric. Then $\R^3\setminus K$ is path connected, $K$ has zero three-dimensional Lebesgue measure, and
\begin{equation}
 d_{g,\R^3\setminus K}(x,y)=d_g(x,y)
 \qquad(x,y\notin K).
 \label{eq:distance-equality}
\end{equation}
In particular, if $K\Subset\Omega$ and $x\notin K$, then
\begin{equation}
 d_{g,\R^3\setminus K}(x,E)=r(x).
 \label{eq:distance-exterior}
\end{equation}
\end{lemma}

\begin{proof}
Set $D=\R^3\setminus K$. Each Lipschitz curve has zero three-dimensional measure, so $K$ does too. Restricting paths gives $d_{g,D}\geq d_g$, and it remains to construct paths in $D$ with lengths arbitrarily close to $d_g$.

Fix $p,q\in D$ and let $m=(p+q)/2$. If $[p,z]$ meets a curve $b([0,1])\subset K$, then $z=p+\alpha(b(\theta)-p)$, $\theta\in[0,1],\ \alpha\geq1$. For $1\leq\alpha\leq j$, this is a Lipschitz image of a two-dimensional rectangle and has zero three-dimensional measure.
The countable union over integers $j$ and finite union over the curves therefore give a null set $C_p$ of bad points. The analogous set $C_q$ is also null. Every ball about $m$ contains a point $z\notin C_p\cup C_q$, so $[p,z]\cup[z,q]\subset D$.
Continuity of segment length in its endpoints gives $\ell_g([p,z])+\ell_g([z,q])\longrightarrow\ell_g([p,q]) \ (z\to m)$.

Given $x,y\in D$ and $\eta>0$, approximate a nearly shortest path by a polygon of length less than $d_g(x,y)+\eta/2$.
Its finitely many interior vertices can be perturbed into $D$ while preserving this strict bound, since $K$ has zero measure and segment lengths are continuous. If the polygon has $N$ segments, replace each by the avoiding two-segment path above with extra length less than $\eta/(2N)$. The resulting path lies in $D$ and has length less than $d_g(x,y)+\eta$. This proves path connectedness, and letting $\eta\downarrow0$ proves \eqref{eq:distance-equality}.
Finally, $K\Subset\Omega$ gives $E\subset D$, hence
\[
 d_{g,D}(x,E)
 =\inf_{y\in E}d_{g,D}(x,y)
 =\inf_{y\in E}d_g(x,y)
 =d_g(x,E)=r(x).
\]
\end{proof}

For time-independent wave operators, unique continuation gives the following.
If $M$ is a connected smooth Riemannian manifold, $V\in L^\infty_{\loc}(M)$, and 
\[
(\partial_t^2-\Delta_g+V)z=0,
 \qquad z\in H^1_{\loc}(M\times(-S,S)),
\]
 then vanishing on $\omega\times(-S,S)$ for a nonempty open set $\omega\subset M$ implies vanishing where
\begin{equation}
 d_{g,M}(x,\omega)+|t|<S.
 \label{eq:tataru}
\end{equation}
This follows from \cite[Theorem~3.24]{LL2023}, based on the local theorem of Tataru~\cite{Tataru1995}. Applying this result on suitably shifted time intervals yields the following one-sided continuation property for $P$.

\begin{lemma}\label{lem:one-sided}
Let $D\subset\R^3$ be a connected open set, $\omega\subset D$ a nonempty open set, and $b\in\R$. Suppose 
\[
Pv=0\quad\text{in }D\times\R,\qquad
 v\in C^\infty(D\times\R),\qquad
 v=0\quad\text{on }\omega\times(-\infty,b).
\]
 Then
\begin{equation}
 v(x,t)=0\quad\text{if }t+d_{g,D}(x,\omega)<b.
 \label{eq:one-sided}
\end{equation}
\end{lemma}

\begin{proof}
For $g=n\dd x^2$ in dimension three, $\Delta_g=n^{-1}\Delta+\frac{1}{2}n^{-2}\nabla n\cdot\nabla$. Set $\varpi=n^{1/4}$ and $v=\varpi z$. A direct computation gives
\begin{equation}
 \varpi^{-1}n^{-1}P(\varpi z)
 =\partial_t^2z-\Delta_g z+Vz,
 \qquad V=-n^{-5/4}\Delta(n^{1/4}).
 \label{eq:conjugation}
\end{equation}
The potential is smooth. Fix $(x,t)$ satisfying \eqref{eq:one-sided} and choose $d_{g,D}(x,\omega)<S<b-t$. The time interval $(t-S,t+S)$ is contained in $(-\infty,b)$. Apply \eqref{eq:tataru} to $z(x,t+\tau)$ on $-S<\tau<S$ and evaluate at $\tau=0$. This proves the assertion.
\end{proof}

We now combine these facts. The argument first identifies the source in the visible region
\begin{equation}
 \mathcal V_T=\{(x,t):t+r(x)<T\}.
 \label{eq:visible-region}
\end{equation}

\begin{proof}[Proof of Theorem~\ref{thm:main}]
Set $w=u_1-u_2$, $F=F_1-F_2$, and
\begin{equation}
 K=a_1([0,T_0])\cup a_2([0,T_0]),\qquad
 D=\R^3\setminus K.
 \label{eq:punctured-domain}
\end{equation}
By Lemma~\ref{lem:distance}, $K$ is compact and has zero three-dimensional measure, and $D$ is connected.
Using the mollifier in \eqref{eq:mollification}, define $w_\eps=\rho_\eps*_tw$ and $F_\eps=\rho_\eps*_tF$.
Lemma~\ref{lem:smoothing} gives $Pw_\eps=F_\eps$,
$w_\eps\in C^\infty(\R;L^2(\R^3))$, and $w_\eps\in C^\infty(D\times\R)$.

Finite propagation makes the pressure traces vanish near $t=0$. Together with causality and the measured equality, this gives $w|_\Gamma=0$ for $t<T$.
Since the boundary trace commutes with temporal convolution, $w_\eps|_\Gamma=0$ for $t<T-\eps$.
The exterior wave equation has zero initial and Dirichlet boundary data. The wave energy estimate
\cite[Chapter~6]{Taylor2011}, applied on each connected component of $E$, therefore gives $w_\eps=0$ on $E\times(-\infty,T-\eps)$.

Since $Pw_\eps=0$ on $D\times\R$, Lemmas~\ref{lem:distance} and~\ref{lem:one-sided} imply
\begin{equation}
 w_\eps(x,t)=0
 \qquad\text{if }x\notin K
 \text{ and }t+r(x)<T-\eps.
 \label{eq:off-curve-zero}
\end{equation}
Because $w_\eps$ is locally integrable and $K\times\R$ has zero measure, it vanishes distributionally throughout $\mathcal V_{T-\eps}$. Consequently, $F_\eps=Pw_\eps=0$ there.

For any $\varphi\in C_c^\infty(\mathcal V_T)$, compactness and continuity of $r$ give $t+r(x)\leq T-\delta$ on $\supp\varphi$
for some $\delta>0$.
Thus $\langle F_\eps,\varphi\rangle=0$ for $\eps<\delta$.
Since $F_\eps\to F$ in the sense of distributions, we obtain $F=0$ in $\mathcal V_T$.

If $T\geq T_0+R_0$, then 
\[
t+r(a_j(t))\leq t+R_0<T_0+R_0\leq T
 \qquad(j=1,2,\ 0<t<T_0).
\]
 Hence $(a_j(t),t)\in\mathcal V_T$ for $j=1,2$ and $0<t<T_0$.
Since $\lambda_j\in L^1(0,T_0)$, the sources have no mass at $t=0$ or $t=T_0$.
It follows that $F_1=F_2$.

Choose $\chi\in C_c^\infty(\R^3)$ equal to one near $D_0$. Testing $F_1=F_2$ against $\chi(x)\eta(t)$ and $x_m\chi(x)\eta(t)$,
for $\eta\in C_c^\infty(0,T_0)$ and $m=1,2,3$, gives
\begin{equation}
 \lambda_1=\lambda_2=:\lambda,\qquad
 \lambda(t)(a_1(t)-a_2(t))=0
 \quad\text{a.e. in }(0,T_0).
 \label{eq:mass-equality}
\end{equation}
Continuity of the trajectories implies $a_1=a_2$ on $\supp(|\lambda|\dd t)$.
By \eqref{eq:activity}, this support is $[0,T_0]$.
Thus \eqref{eq:atomic-equality} holds, completing the proof.
\end{proof}

\begin{remark}\label{rem:multiple-source}
The same argument can be extended to a finite sum of moving point sources
\[
 F_j(x,t)=\sum_{\ell=1}^{N_j}\lambda_{j,\ell}(t)
                   \delta_{a_{j,\ell}(t)}(x),\qquad j=1,2,
\]
where each pair satisfies \eqref{eq:admissible}. The union of their spatial trajectory traces is still a finite union of compact Lipschitz curve images,
so the same regularization and continuation argument applies.
Hence equality of the boundary pressure traces on $\Gamma\times(0,T)$ implies $F_1=F_2$ in $\D'(\mathcal V_T)$.
If $T\geq T_0+R_0$, the complete source distributions coincide.

This determines the total source distribution.
However, collisions, cancellation, and ambiguity in source labeling may prevent unique recovery of the individual trajectories and intensities.
\end{remark}

\section{Inverse moving source problem: four-point measurements}\label{sec:four-stability}
\label{sec:stability}
Whole-boundary measurements are often difficult to obtain in practice, and only finitely many receivers are available.
We therefore study conditional joint Lipschitz stability from pressure waveforms at four fixed point receivers in this section and develop reconstruction methods in Section~\ref{sec:algorithms}.

Consider a single moving point source with real-valued intensity in a compact convex region $D\Subset\Omega$ and four distinct fixed receivers $x_1,\ldots,x_4\in\R^3\setminus D$.
Receivers on $\Gamma$ are included as a special case.
Let $G$ be the causal Green function normalized by $P_xG(x,y,t)=\delta_y(x)\delta_0(t)$.
For a fixed time interval $(0,H)$, we impose the following propagation conditions on the medium and receiver configuration:
\begin{enumerate}
\renewcommand{\labelenumi}{(\roman{enumi})}
\item For $y\in D$ and $0<t<H$, the Green function admits the exact distributional representation
\begin{equation}
G(x_i,y,t)=A_i(y)\delta(t-\tau_i(y))
+\one_{\{t\geq\tau_i(y)\}}V_i(y,t),\qquad i=1,\ldots,4.
\label{eq:Green-hypothesis}
\end{equation}
Here $\tau_i=d_g(\cdot,x_i)>0$ is $C^2$ near $D$, $A_i$ is $C^1$ near $D$ with $A_i\geq A_*>0$, and $V_i$ has a $C^1$ extension near $D\times[0,H]$, with fixed bounds on these norms.
\item The travel-time geometry is uniformly nondegenerate:
\begin{equation}
Q(y)=\begin{pmatrix}
\nabla\tau_1(y)^T&1\\
\nabla\tau_2(y)^T&1\\
\nabla\tau_3(y)^T&1\\
\nabla\tau_4(y)^T&1
\end{pmatrix},\qquad
\inf_{y\in D}\sigma_{\min}(Q(y))\geq\kappa>0.
\label{eq:travel-rank}
\end{equation}
\end{enumerate}
Here $\sigma_{\min}(Q(y))$ denotes the smallest singular value of $Q(y)$.
Representation \eqref{eq:Green-hypothesis} allows a $C^1$ scattered tail but excludes additional singular arrivals at the receivers before time $H$.

For fixed $q_0\in C_c^\infty(\Omega)$ and $H>\max_{i,y\in D}|x_i-y|$, the medium $n_\eps=1+\eps q_0$ satisfies \eqref{eq:Green-hypothesis} for all sufficiently small $|\eps|$.
The rank bound \eqref{eq:travel-rank} is also preserved under sufficiently small perturbations, provided it holds in the homogeneous reference medium.
Both assertions are proved in Proposition~\ref{prop:small-medium} in Appendix~\ref{app:Green}.
For example, the Euclidean rank condition holds when the receivers lie at the vertices of a regular tetrahedron and $D$ is a sufficiently small ball around the center of the tetrahedron.

We further impose the following conditions on the source and observation intervals:
\begin{enumerate}
\renewcommand{\labelenumi}{(\roman{enumi})}
\setcounter{enumi}{2}
\item For fixed $a_-,a_+\in D$ and positive constants $M_a,M_\lambda,\lambda_*,\theta$, the source satisfies
\begin{equation}
\begin{gathered}
a\in W^{2,\infty}(0,T_0;\R^3),\quad a([0,T_0])\subset D,\quad
\|a\|_{W^{2,\infty}}\leq M_a,\quad a(0)=a_-,\quad a(T_0)=a_+,\\
\lambda\in W^{1,\infty}(0,T_0),\quad
\|\lambda\|_{W^{1,\infty}}\leq M_\lambda,\quad
\lambda(s)\geq\lambda_*>0,\\
\beta_i(s):=1+\nabla\tau_i(a(s))\cdot a'(s)\geq\theta>0
\quad(1\leq i\leq4).
\end{gathered}
\label{eq:four-prior}
\end{equation}
\item The time interval and receiver observation intervals satisfy
\begin{equation}
H>T_0+\max_{i,y\in D}\tau_i(y),\qquad
I_i=(\tau_i(a_-),\,T_0+\tau_i(a_+)).
\label{eq:four-time}
\end{equation}
\end{enumerate}
The bound $\sqrt{n(a(s))}|a'(s)|\leq1-\theta$ implies $\beta_i(s)\geq\theta$.
The prescribed endpoints ensure that all admissible sources have the same arrival intervals, although the waveforms may have jumps at the interval endpoints.
The measurement window must contain every $I_i$, which may require a longer observation time than the threshold in Theorem~\ref{thm:main}.

\begin{theorem}\label{thm:four-Lip}
Let $x_1,\ldots,x_4\in\R^3\setminus D$ be four distinct fixed receivers, and assume \eqref{eq:Green-hypothesis}--\eqref{eq:four-time} hold.
Let $(a_j,\lambda_j)$, $j=1,2$, be two admissible source pairs whose trajectories have the common prescribed endpoints $a_j(0)=a_-$ and $a_j(T_0)=a_+$, and let $u_j$ be their causal fields.
Then there exists a constant $C>0$, depending only on the prescribed bounds on $\tau_i,A_i,V_i$, $\kappa,A_*,\theta,\lambda_*,M_a,M_\lambda,D$ and $T_0$, such that
\begin{equation}
 \|a_1-a_2\|_{W^{1,\infty}(0,T_0)}
   +\|\lambda_1-\lambda_2\|_{L^\infty(0,T_0)}
 \leq
 C\max_{1\leq i\leq4}
 \|u_1(x_i,\cdot)-u_2(x_i,\cdot)\|_{L^\infty(I_i)}.
 \label{eq:four-Lip}
\end{equation}
\end{theorem}

We first derive an exact Volterra identity for the receiver waveforms.
Set 
\[
r_{ji}(s)=s+\tau_i(a_j(s)),\qquad
\beta_{ji}(s)=r_{ji}'(s),\qquad
w_{ji}(t)=u_j(x_i,t).
\]
 By \eqref{eq:four-prior}, $r_{ji}$ is a strictly increasing bi-Lipschitz map from $[0,T_0]$ onto $\overline I_i$.
The prescribed endpoints ensure that this arrival interval is the same for both sources.
Substituting \eqref{eq:Green-hypothesis} into the Green representation of $u_j$ gives, for $0<s<T_0$,
\begin{equation}
\begin{split}
w_{ji}(r_{ji}(s))
&=\frac{\lambda_j(s)A_i(a_j(s))}{\beta_{ji}(s)}+R_{ji}(s),\\
R_{ji}(s)
&=\int_0^s\lambda_j(\sigma)V_i(a_j(\sigma),r_{ji}(s)-\sigma)\dd\sigma.
\end{split}
\label{eq:Volterra-identity}
\end{equation}
The direct term follows from the change-of-variables formula for the Dirac distribution.
For the tail, strict monotonicity makes the Heaviside condition $\sigma+\tau_i(a_j(\sigma))\leq r_{ji}(s)$ equivalent to $\sigma\leq s$.

Define the tail-subtracted quantity by
\begin{equation}
p_{ji}(s)=w_{ji}(r_{ji}(s))-R_{ji}(s)
=\frac{\lambda_j(s)A_i(a_j(s))}{\beta_{ji}(s)}.
\label{eq:direct-pressure}
\end{equation}
Both $p_{ji}$ and $\beta_{ji}$ are uniformly bounded above and away from zero, with bounds determined by the above assumptions.

\begin{lemma}\label{lem:waveform-Lip}
Each waveform $w_{ji}$ is Lipschitz continuous on $I_i$ with a uniform constant $L_w$ and extends continuously to $\overline I_i$.
\end{lemma}

\begin{proof}
The derivative of the arrival factor satisfies 
\[
\beta_{ji}'=(a_j')^TD^2\tau_i(a_j)a_j'
+\nabla\tau_i(a_j)\cdot a_j''
\quad\text{a.e.}
\]
 The assumptions therefore give a uniform bound on $p_{ji}'$.
Since $\beta_{ji}\geq\theta$, the observation-time derivative of $p_{ji}\circ r_{ji}^{-1}$ is also uniformly bounded.
For the tail, put $s=r_{ji}^{-1}(t)$.
Differentiation of its integral gives
\begin{equation}
\frac{\dd}{\dd t}R_{ji}(r_{ji}^{-1}(t))
=\frac{\lambda_j(s)V_i(a_j(s),\tau_i(a_j(s)))}{\beta_{ji}(s)}
+\int_0^s\lambda_j(r)\partial_tV_i(a_j(r),t-r)\dd r.
\label{eq:tail-derivative}
\end{equation}
Both terms are uniformly bounded.
Together with \eqref{eq:Volterra-identity}, this proves the uniform Lipschitz bound on $I_i$.
The continuous extension to $\overline I_i$ follows.
\end{proof}

We now turn to the proof of Theorem ~\ref{thm:four-Lip}.
\begin{proof}
Set 
\[
e_a=a_1-a_2,\qquad \mu=\lambda_1-\lambda_2,\qquad
\delta=\max_i\|w_{1i}-w_{2i}\|_{L^\infty(I_i)}.
\]
 For $0<s<T_0$, both $r_{1i}(s)$ and $r_{2i}(s)$ lie in $I_i$.
Using the continuous representatives of the waveforms and Lemma~\ref{lem:waveform-Lip}, we obtain
\begin{equation}
\begin{split}
|w_{1i}(r_{1i}(s))-w_{2i}(r_{2i}(s))|
&\leq\delta+L_w|r_{1i}(s)-r_{2i}(s)|\\
&\leq\delta+C|e_a(s)|.
\end{split}
\label{eq:wave-compare}
\end{equation}
The $C^1$ extension of $V_i$ on a neighborhood of the compact product $D\times[0,H]$, convexity of $D$, and \eqref{eq:four-time} give
\[
|R_{1i}(s)-R_{2i}(s)|
\leq C|e_a(s)|
+C\int_0^s\big(|e_a(r)|+|\mu(r)|\big)\dd r.
\]
Indeed, the two kernel time arguments differ by $\tau_i(a_1(s))-\tau_i(a_2(s))$, and their source arguments differ by $e_a(r)$.
All the compared arguments remain in the stated compact product.
It follows that
\begin{equation}
|p_{1i}(s)-p_{2i}(s)|
\leq\delta+C|e_a(s)|
+C\int_0^s\big(|e_a(r)|+|\mu(r)|\big)\dd r.
\label{eq:p-compare}
\end{equation}

For shorter notation, write $A_{ji}=A_i(a_j(s))$, $b_{ji}=\nabla\tau_i(a_j(s))$, and $v_j=a_j'(s)$.
Subtract the two identities $p_{ji}(1+b_{ji}\cdot v_j)=\lambda_jA_{ji}$.
Keeping the derivative difference on the left gives the exact equation
\begin{equation}
p_{1i}b_{1i}\cdot e_a'-A_{1i}\mu
=\lambda_2(A_{1i}-A_{2i})
-p_{1i}(b_{1i}-b_{2i})\cdot v_2-(p_{1i}-p_{2i})\beta_{2i}.
\label{eq:exact-difference}
\end{equation}
The matrix on the left, with rows $[p_{1i}b_{1i}^T,-A_{1i}]$, has the factorization
\begin{equation}
M_1=
\operatorname{diag}\left(\frac{A_{1i}}{\beta_{1i}}\right)_{i=1}^4
Q(a_1)
\begin{pmatrix}\lambda_1 I_3&-v_1\\0&-1\end{pmatrix}.
\label{eq:matrix-factor}
\end{equation}
The lower bounds on $A_{1i}$ and $\lambda_1$, the upper bounds on $\beta_{1i}$ and $v_1$, and \eqref{eq:travel-rank} give uniformly bounded inverses for all three factors.
Hence $\|M_1^{-1}\|\leq C$.

Combining \eqref{eq:exact-difference}, \eqref{eq:matrix-factor}, and \eqref{eq:p-compare} gives, for almost every $s$,
\begin{equation}
z(s):=|e_a'(s)|+|\mu(s)|
\leq C\delta+C|e_a(s)|
+C\int_0^s\big(|e_a(r)|+|\mu(r)|\big)\dd r.
\label{eq:z-inequality}
\end{equation}
Since $e_a(0)=0$, we have $|e_a(s)|\leq\int_0^s z(r)\dd r$.
Consequently, 
\[
z(s)\leq C\delta+C(1+T_0)\int_0^s z(r)\dd r.
\]
 Gronwall's inequality gives $\|z\|_{L^\infty(0,T_0)}\leq C\delta$.
Integrating $e_a'$ then yields $\|e_a\|_{L^\infty(0,T_0)}\leq C\delta$, which proves \eqref{eq:four-Lip}.
\end{proof}

\begin{remark}\label{rem:four-endpoints}
In the present proof, the common initial position is used in the Gronwall estimate, and the prescribed common endpoints ensure that the receiver arrival intervals coincide.
The waveform comparison is therefore performed within the same intervals, where the waveforms are uniformly Lipschitz continuous.
\end{remark}

\section{Reconstruction methods}\label{sec:algorithms}

For the fixed receiver configuration of Section~\ref{sec:four-stability}, we develop a Volterra--Euler reconstruction scheme based on \eqref{eq:Volterra-identity} and establish an error estimate under the assumptions of Theorem~\ref{thm:four-Lip}.
We also present a spectral reconstruction method based on the full outgoing Green function for any known coefficient satisfying \eqref{eq:medium}. It does not require the time-domain propagation assumptions of the Volterra--Euler scheme. Numerical experiments with time-domain and frequency-domain data demonstrate the effectiveness of both reconstruction methods.

\subsection{Volterra--Euler scheme}

At each emission-time step, we evaluate the tail using the reconstructed source history and solve a linear system for the current velocity and intensity.
The trajectory is then advanced by an Euler step.

Assume \eqref{eq:Green-hypothesis}--\eqref{eq:four-time} hold on a compact convex region $K\Subset\Omega$ containing the true trajectory.
We further require $\dist(a(s),\partial K)\geq\rho>0$ for $0\leq s\leq T_0$ and 
\[
\tau_{\min}=\min_{i,y\in K}\tau_i(y)>0,\qquad
T_0+\max_{i,y\in K}\tau_i(y)<H.
\]
 Write $|a'|\leq M_v$ and $\lambda_*\leq\lambda\leq\lambda^*$, retaining the derivative and arrival-factor bounds in \eqref{eq:four-prior}.

To allow tail evaluations for $t<\tau_i(y)$, define
\begin{equation}
V_i^e(y,t)=V_i(y,\max\{t,\tau_i(y)\}),\qquad (y,t)\in K\times[0,H].
\label{eq:tail-extension}
\end{equation}
This extension is Lipschitz continuous on $K\times[0,H]$.
Let $\widehat\tau_i$, $\widehat A_i$, and $\widehat V_i^e$ be approximations, with $\widehat\tau_i$ of class $C^1$ near $K$.
Assume that
\begin{equation}
\max_i\left(
\|\widehat\tau_i-\tau_i\|_{C^1(K)}
+\|\widehat A_i-A_i\|_{C^0(K)}
+\|\widehat V_i^e-V_i^e\|_{C^0(K\times[0,H])}
\right)\leq\epsilon_K.
\label{eq:kernel-error}
\end{equation}

For a single source, write $w_i(t)=u(x_i,t)$.
Suppose that samples are acquired on a time grid covering each arrival interval, with maximum spacing $\Delta$:
\begin{equation}
y_{ik}=w_i(t_{ik})+\xi_{ik},\qquad |\xi_{ik}|\leq\delta.
\label{eq:sample-noise}
\end{equation}
For each receiver, retain all samples selected by the cropping rule in Appendix~\ref{app:interior-sampling}, and require at least two retained samples.
Let $J_i$ be their piecewise linear interpolant, extended constantly beyond the first and last retained sample.
Let $W_i$ agree with $w_i$ on $I_i$ and extend constantly outside $I_i$ using its interior endpoint limits.
Lemma~\ref{lem:interpolation} gives $\|J_i-W_i\|_{L^\infty(\R)} \leq\eta:=\delta+L_w(\Delta+2\epsilon_K)$. The $2\epsilon_K$ term is omitted when the arrival interval endpoints are known exactly.

Set $h=T_0/N$, $s_n=nh$, and $a_0=a_-$.
For $n=0,\ldots,N-1$, compute
\begin{equation}
\begin{split}
\widehat r_i^n&=s_n+\widehat\tau_i(a_n),\\
R_i^n&=h\sum_{j=0}^{n-1}\ell_j
\widehat V_i^e(a_j,\widehat r_i^n-s_j),\\
p_i^n&=J_i(\widehat r_i^n)-R_i^n.
\end{split}
\label{eq:discrete-tail}
\end{equation}
The empty sum is understood as zero.
Set $p^n=(p_1^n,\ldots,p_4^n)^T$ and solve
\begin{equation}
\begin{gathered}
(M_n)_{i,:}=[p_i^n\nabla\widehat\tau_i(a_n)^T,-\widehat A_i(a_n)],\\
M_n\binom{v_n}{\ell_n}=-p^n+e_n^{\rm lin},
\quad |e_n^{\rm lin}|\leq\epsilon_{\rm lin},\qquad
a_{n+1}=a_n+hv_n.
\end{gathered}
\label{eq:discrete-system}
\end{equation}
Here $a_n$, $v_n$, and $\ell_n$ approximate $a(s_n)$, $a'(s_n)$, and $\lambda(s_n)$, respectively.
The linear system is solved using QR or SVD, and the residual $e_n^{\rm lin}$ is measured in the original unscaled system.

At each step, we impose the following checks.
Before solving, require $p_i^n\geq p_*/2$ for all $i$ and $\sigma_{\min}(M_n)\geq m_*/2$, where $p_*,m_*>0$ are defined in Appendix~\ref{app:guard-thresholds}.
After solving, require
\begin{equation}
\lambda_*/2\leq\ell_n\leq2\lambda^*,\qquad
|v_n|\leq M_v+1,\qquad
1+\nabla\widehat\tau_i(a_n)\cdot v_n\geq\theta/2
\quad(1\leq i\leq4).
\label{eq:post-guards}
\end{equation}
All spatial evaluation points and the new position $a_{n+1}$ must lie in $K$, and all time arguments of $\widehat V_i^e$ must lie in $[0,H]$.
Stop if any of these checks fails.
Direct evaluation of the history sums requires $O(N^2)$ tail evaluations and $O(N)$ storage.

Define the reconstructed functions by
\[
\begin{aligned}
a_h(s)&=a_n+(s-s_n)v_n,\\
\lambda_h(s)&=\ell_n,
\end{aligned}
\qquad s_n\leq s<s_{n+1},\qquad a_h(T_0)=a_N.
\]
Thus $a_h$ is piecewise linear and $\lambda_h$ is piecewise constant.

\begin{theorem}\label{thm:march-error}
Under the assumptions above, with the sampling and safeguard rules specified in Appendix~\ref{app:marching}, there exist constants $C,\varepsilon_0>0$, depending only on the prescribed propagation and source bounds, the spatial and temporal margins, and $T_0$, such that if
\begin{equation}
q_{\rm err}:=h+\delta+\Delta+\epsilon_K+\epsilon_{\rm lin}\leq\varepsilon_0,
\label{eq:error-budget}
\end{equation}
the scheme completes all steps without a failed check and
\begin{equation}
\|a_h-a\|_{W^{1,\infty}(0,T_0)}
+\|\lambda_h-\lambda\|_{L^\infty(0,T_0)}
\leq Cq_{\rm err}.
\label{eq:march-error}
\end{equation}
For sufficiently small $\varepsilon_0$, the reconstructed arrival times satisfy
\[
\frac{\dd}{\dd s}\big(s+\widehat\tau_i(a_h(s))\big)\geq\theta/2
\quad\text{for almost every }s\in(0,T_0),\qquad i=1,\ldots,4.
\]
\end{theorem}
The proof is given in Appendix~\ref{app:marching}.
The scheme does not impose $a_N=a_+$, but \eqref{eq:march-error} also gives $|a_N-a_+|\leq Cq_{\rm err}$.

\subsection{Spectral reconstruction}\label{sec:batch}

Let $G_k$ be the outgoing Green function satisfying $(\Delta_x+k^2n)G_k(x,y)=-\delta_y(x)$.
With the Fourier convention $e^{ikt}$, the predicted receiver spectra at the selected positive frequencies $k_m$ are
\begin{equation}
\mathcal F_{im}(a,\lambda)
=\int_0^{T_0}\lambda(s)e^{ik_ms}G_{k_m}(x_i,a(s))\dd s,
\qquad i=1,\ldots,4.
\label{eq:spectral-forward}
\end{equation}
These are spectra of the complete receiver waveforms.
Using the full outgoing Green function accounts for both direct propagation and scattering by the medium.

We represent the trajectory and intensity by clamped cubic B-splines, allowing different numbers of basis functions: 
\[
a_c(s)=\sum_j P_jB_j^a(s),\qquad
\lambda_\ell(s)=\sum_j\ell_jB_j^\lambda(s).
\]
 Here $P_j\in\R^3$, $c$ collects the free trajectory control coordinates, and $\ell=(\ell_j)_j$ is the intensity coefficient vector.
We fix the first and last trajectory control points at $a_-$ and $a_+$, constrain the remaining control points to a prescribed box containing these endpoints, and impose $\lambda_*\leq\ell_j\leq\lambda^*$.
The convex-hull property keeps the trajectory in this box and ensures $\lambda_*\leq\lambda_\ell(s)\leq\lambda^*$ throughout the emission interval~\cite{deBoor2001}.

Let $\mathcal F$ collect the predicted spectra in \eqref{eq:spectral-forward}, and let $y\neq0$ denote the measured spectra.
With regularization parameters $\alpha_a,\alpha_\lambda>0$, we minimize
\begin{equation}
\Phi(c,\ell)
=\frac{\|\mathcal F(a_c,\lambda_\ell)-y\|_2^2}{2\|y\|_2^2}
+\frac{\alpha_a}{2}\int_0^{T_0}|a_c''(s)|^2\dd s+\frac{\alpha_\lambda}{2}\int_0^{T_0}|\lambda_\ell'(s)|^2\dd s.
\label{eq:spectral-objective}
\end{equation}
The last two terms penalize rapid changes in the velocity and intensity~\cite{EHN1996}.

For a fixed trajectory, the forward map is linear in $\ell$, so we use bounded variable projection~\cite{GP2003}:
\begin{equation}
\ell_*(c)\in\mathop{\rm argmin}_{\lambda_*\leq\ell_j\leq\lambda^*}
\Phi(c,\ell),\qquad
\min_c\Phi(c,\ell_*(c)).
\label{eq:variable-projection}
\end{equation}
The inner problem is solved by BVLS~\cite{SP1995}, using a real linear least-squares system formed from the real and imaginary data components and the intensity regularization.
A bound-constrained trust-region reflective method~\cite{BCL1999} updates the trajectory coefficients.

Differentiating \eqref{eq:spectral-forward} gives
\begin{equation}
D\mathcal F_{im}(a,\lambda)(\delta a,\delta\lambda)
=\int_0^{T_0}e^{ik_ms}
\Big(\delta\lambda(s)G_{k_m}(x_i,a(s))+\lambda(s)\nabla_yG_{k_m}(x_i,a(s))\cdot\delta a(s)\Big)\dd s.
\label{eq:spectral-Jacobian}
\end{equation}
Here $\delta a$ and $\delta\lambda$ are perturbations of the trajectory and intensity, and $\nabla_y$ differentiates with respect to the source position.
For a locally unchanged active set, coefficients at their bounds have zero derivative.
We compute the derivatives of the remaining coefficients from their regularized normal equations, assuming this system is nonsingular.
The resulting reduced Jacobian includes the residual-dependent term, and the derivative systems are solved using QR factorization.

We initialize the reconstruction using a coarse spline fit at low frequencies and transfer the result to the finer spline space.
Regularization parameter selection is described in Section~\ref{sec:numerics}.
The method computes a local fit within the prescribed spline class.

\subsection{Numerical experiments}\label{sec:numerics}

We present two sets of experiments with four receivers. The first examines the Volterra--Euler scheme using time-domain waveforms in a medium whose finite-time Green function is known explicitly. The second illustrates regularized spectral reconstruction from discrete frequency-domain data in a known inhomogeneous medium, without imposing \eqref{eq:Green-hypothesis}. We set $T_0=4$, and the receivers are
\[
 (x_1,x_2,x_3,x_4)=\frac{2}{\sqrt3}
 \big((1,1,1),(1,-1,-1),(-1,1,-1),(-1,-1,1)\big).
\]
Both experiments use relative noise levels $0\%,3\%,5\%,10\%$. We measure the reconstruction errors by
\begin{equation}
 E_a=\frac{\|a_{\rm rec}-a\|_{L^2(0,T_0)}}{\|a\|_{L^2(0,T_0)}},\qquad
 E_\lambda=\frac{\|\lambda_{\rm rec}-\lambda\|_{L^2(0,T_0)}}
                       {\|\lambda\|_{L^2(0,T_0)}}.
 \label{eq:numerical-errors}
\end{equation}
For the Volterra--Euler output, the integrals are evaluated by eight-point Gauss quadrature on each time step. For the spectral output, we use discrete norms at 2001 equispaced emission times. Below, $u=s/4$ and $\vartheta=2\pi(u+1/8)$.

\medskip\noindent
\textbf{Volterra--Euler reconstruction with a verified finite-time Green function.}
We consider
\begin{equation}
 f(x)=1+\frac{|x|^2}{4R^2},\qquad
 n(x)=1+\chi(|x|)\bigl(f(x)^{-2}-1\bigr),\qquad R=7.5,
 \label{eq:volterra-numerical-medium}
\end{equation}
where $0\leq\chi\leq1$ is a smooth cutoff equal to one for $|x|\leq4.75$ and zero for $|x|\geq5.25$. The coefficient satisfies $n\geq f(5.25)^{-2}>0.79364$; numerically, $\min n\simeq0.82199$. On the source region $K=\overline B_{0.35}(0)$, $|n-1|\leq1.089\times10^{-3}$. We take $H=6.6$.

Put $d(r)=2R\arctan(r/(2R))$. Since $n\geq f^{-2}$, any influence of the cutoff region on a receiver must travel for at least 
\[
2d(4.75)-d(0.35)-d(2)>6.86208>H.
\]
 Moreover, $T_0+\max_{i,y\in K}\tau_i(y)<6.33821<H<\pi R$. Finite propagation therefore gives the Green function of the spherical metric $f^{-2}\dd x^2$ at all four receivers throughout the observation interval, with no additional singular arrival. With $b_i(y)=[f(x_i)f(y)]^{-1/2}$, its coefficients are
\begin{equation}
\begin{aligned}
 \tau_i(y)&=2R\arcsin\frac{|x_i-y|b_i(y)}{2R},\qquad
 A_i(y)=\frac{b_i(y)}{4\pi R\sin(\tau_i(y)/R)},\\
 V_i(y,t)&=\frac{b_i(y)\tau_i(y)}{32\pi R^3\sin(\tau_i(y)/R)}
 {}_0F_1\!\left(;2;\frac{t^2-\tau_i(y)^2}{16R^2}\right),\qquad t\geq\tau_i(y),
\end{aligned}
\label{eq:volterra-numerical-kernel}
\end{equation}
where ${}_0F_1$ is the hypergeometric limit function. The geometry satisfies $\inf_{y\in K}\sigma_{\min}(Q(y))\geq0.72526$. Thus the finite-time Green representation and rank condition are checked directly for this medium.

The waveforms are generated by inverting the arrival maps and evaluating the history integral with 48-point Gauss quadrature, checked against 96 points. Reconstruction uses $N=1024$ marching steps and waveform samples with maximum spacing $\Delta\leq1/4096$. At each receiver, an independent real Gaussian vector is normalized to the prescribed relative discrete $L^2$ noise level. Each source uses one fixed noise direction across the four levels. These percentages differ from the absolute sample-error bound $\delta$ in Theorem~\ref{thm:march-error}. The march starts from $a_-$; the endpoint $a_+$ determines the acquisition intervals but is not imposed on the reconstructed trajectory.

The test trajectories are a rounded planar Z, an open helix, a spatial Lissajous curve, and a spatial trefoil. Their intensities are periodic, increasing, oscillatory, and double-peaked, respectively. The trajectories and intensities are specified below Figures~\ref{fig:ve-reconstruction}--\ref{fig:ve-trefoil-reconstruction}. Here $\zeta$ is the arc-length-timed Z polygon with vertices $(-0.15,0.12)$, $(0.15,0.12)$, $(-0.15,-0.12)$, $(0.15,-0.12)$, rounded by softplus with width $0.04$ in $u$ and corrected affinely to preserve its endpoints.
\par\smallskip\noindent
\begin{minipage}{\linewidth}
\centering
\includegraphics[width=\linewidth]{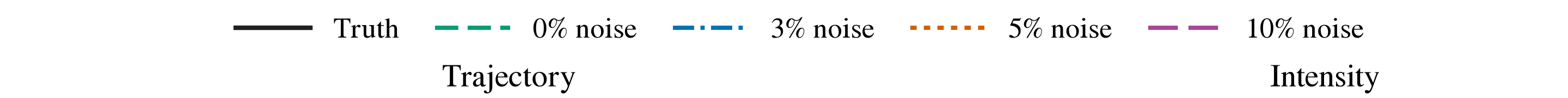}\par
\par\noindent
\begin{minipage}{\linewidth}
\makeatletter
\def\@captype{figure}
\makeatother
\centering
\setlength{\abovecaptionskip}{3pt}
\setlength{\belowcaptionskip}{0pt}
\includegraphics[width=\linewidth]{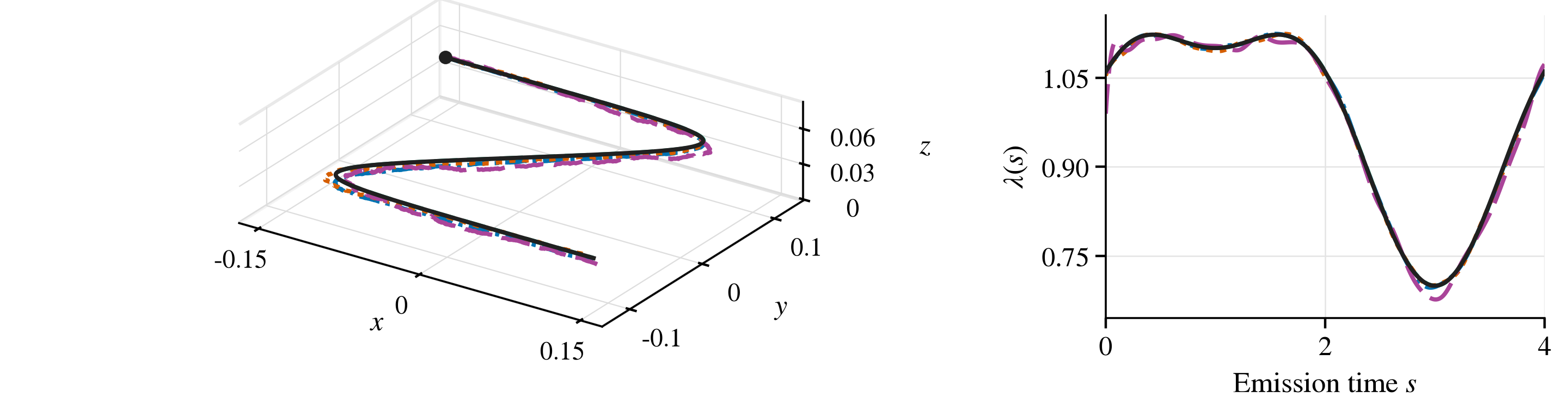}
\caption{Volterra--Euler reconstruction of the rounded-Z source.}
\label{fig:ve-reconstruction}\label{fig:ve-z-reconstruction}
\vspace{1pt}
\noindent\hspace*{0.045\linewidth}%
\begin{minipage}[t][22pt][t]{0.559524\linewidth}
\centering\scriptsize
$\begin{aligned}[t]a_Z(s)&=(\zeta_1(u),\zeta_2(u),0.045),\end{aligned}$\par
\end{minipage}%
\hspace*{0.100714\linewidth}%
\begin{minipage}[t][22pt][t]{0.279762\linewidth}
\centering\scriptsize
\makebox[\linewidth][r]{$\begin{aligned}[t]\lambda_Z(s)&=0.98+0.20\sin(2\pi u)+0.08\cos(4\pi u).\end{aligned}$}\par
\end{minipage}
\end{minipage}
\par\vspace{1pt}
\par\noindent
\begin{minipage}{\linewidth}
\makeatletter
\def\@captype{figure}
\makeatother
\centering
\setlength{\abovecaptionskip}{3pt}
\setlength{\belowcaptionskip}{0pt}
\includegraphics[width=\linewidth]{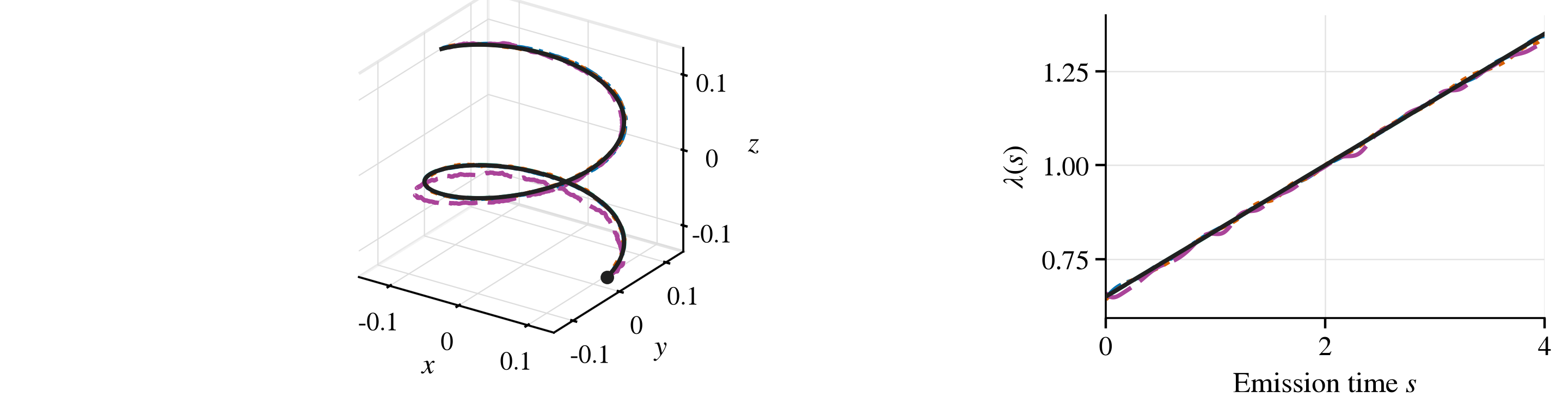}
\caption{Volterra--Euler reconstruction of the open-helix source.}
\label{fig:ve-helix-reconstruction}
\vspace{1pt}
\noindent\hspace*{0.045\linewidth}%
\begin{minipage}[t][22pt][t]{0.559524\linewidth}
\centering\scriptsize
$\begin{aligned}[t]a_H(s)&=(0.12\cos(3\pi u),0.12\sin(3\pi u),-0.12+0.24u),\end{aligned}$\par
\end{minipage}%
\hspace*{0.100714\linewidth}%
\begin{minipage}[t][22pt][t]{0.279762\linewidth}
\centering\scriptsize
\makebox[\linewidth][r]{$\begin{aligned}[t]\lambda_H(s)&=0.65+0.70u.\end{aligned}$}\par
\end{minipage}
\end{minipage}
\par\vspace{1pt}
\par\noindent
\begin{minipage}{\linewidth}
\makeatletter
\def\@captype{figure}
\makeatother
\centering
\setlength{\abovecaptionskip}{3pt}
\setlength{\belowcaptionskip}{0pt}
\includegraphics[width=\linewidth]{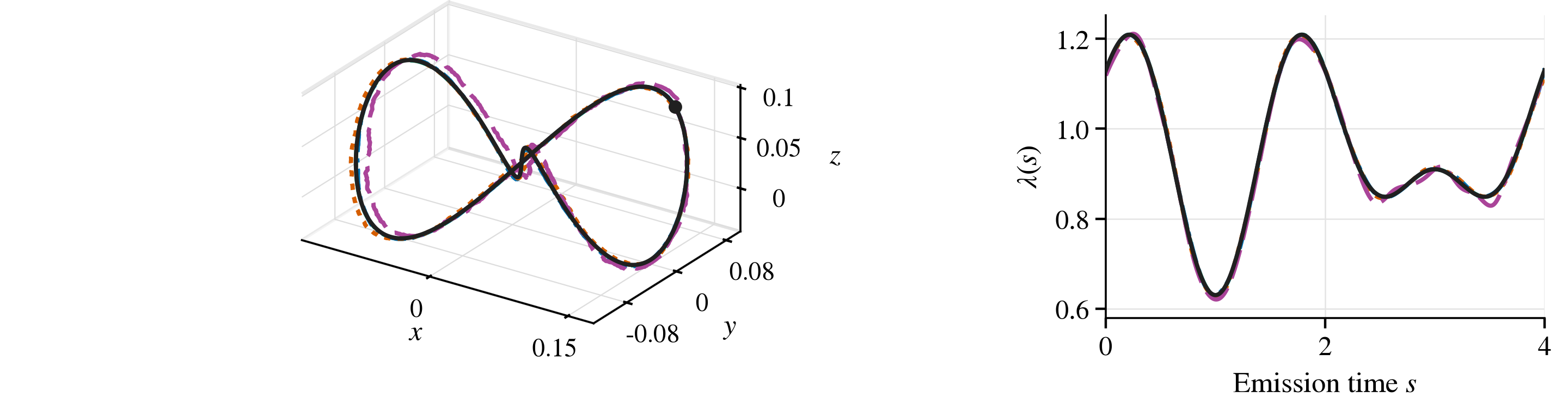}
\caption{Volterra--Euler reconstruction of the spatial Lissajous source.}
\label{fig:ve-lissajous-reconstruction}
\vspace{1pt}
\noindent\hspace*{0.045\linewidth}%
\begin{minipage}[t][22pt][t]{0.559524\linewidth}
\centering\scriptsize
$\begin{aligned}[t]a_L(s)&=(0.02+0.14\sin\vartheta,-0.015+0.10\sin(2\vartheta),0.03+0.06\sin(3\vartheta)),\end{aligned}$\par
\end{minipage}%
\hspace*{0.100714\linewidth}%
\begin{minipage}[t][22pt][t]{0.279762\linewidth}
\centering\scriptsize
\makebox[\linewidth][r]{$\begin{aligned}[t]\lambda_L(s)&=0.95+0.18\cos(4\pi u)+0.14\sin(6\pi u).\end{aligned}$}\par
\end{minipage}
\end{minipage}
\par\vspace{1pt}
\par\noindent
\begin{minipage}{\linewidth}
\makeatletter
\def\@captype{figure}
\makeatother
\centering
\setlength{\abovecaptionskip}{3pt}
\setlength{\belowcaptionskip}{0pt}
\includegraphics[width=\linewidth]{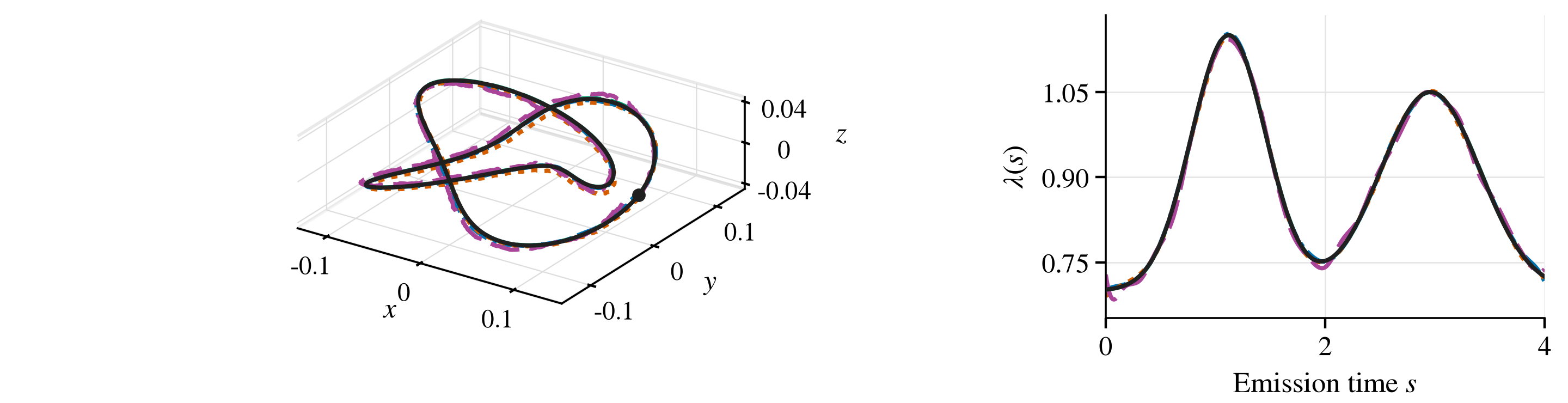}
\caption{Volterra--Euler reconstruction of the spatial trefoil source.}
\label{fig:ve-trefoil-reconstruction}
\vspace{1pt}
\noindent\hspace*{0.045\linewidth}%
\begin{minipage}[t][22pt][t]{0.559524\linewidth}
\centering\scriptsize
$\begin{aligned}[t]a_T(s)&=(r\cos(2\psi),r\sin(2\psi),0.035\sin(3\psi)),\\[-1pt]r&=0.10+0.035\cos(3\psi),\qquad\psi=2\pi u,\end{aligned}$\par
\end{minipage}%
\hspace*{0.100714\linewidth}%
\begin{minipage}[t][22pt][t]{0.279762\linewidth}
\centering\scriptsize
\makebox[\linewidth][r]{$\begin{aligned}[t]\lambda_T(s)&=0.70+0.45e^{-((u-0.28)/0.12)^2}\\[-1pt]&\quad+0.35e^{-((u-0.74)/0.16)^2}.\end{aligned}$}\par
\end{minipage}
\end{minipage}
\par\vspace{1pt}
\end{minipage}
\par\smallskip
These formulas are used only to generate data and evaluate errors. Let $w_{ci}^{\rm dir}$ and $w_{ci}^{\rm tail}$ denote the direct and history terms in \eqref{eq:Volterra-identity} for source $c$ and receiver $i$, on their source-arrival interval $I_{ci}$. The noiseless data satisfy
\[
 \left(\frac{\sum_{c=1}^4\sum_{i=1}^4\|w_{ci}^{\rm tail}\|_{L^2(I_{ci})}^2}
 {\sum_{c=1}^4\sum_{i=1}^4\|w_{ci}^{\rm dir}\|_{L^2(I_{ci})}^2}\right)^{1/2}
 =0.9873\%.
\]
This is a signal norm ratio, rather than a bound on the medium perturbation.

Figures~\ref{fig:ve-reconstruction}--\ref{fig:ve-trefoil-reconstruction} show the reconstructed trajectories on the left and intensities on the right. For readability, the displayed intensities are projected onto a fixed cubic spline space with 29 controls. The trajectories and all reported errors use the raw Volterra--Euler output. All 16 marches complete without a failed safeguard. At $10\%$ noise, the trajectory errors range from $5.06\%$ to $7.63\%$, and the intensity errors from $4.44\%$ to $4.88\%$; see Table~\ref{tab:volterra-reconstruction}.

\begin{table}[!htbp]
\centering\small
\caption{Relative Volterra--Euler reconstruction errors in percent, computed from the raw piecewise linear trajectory and piecewise constant intensity. Column headings give the relative time-domain noise level at each receiver.}
\label{tab:volterra-reconstruction}
\begin{tabular}{lcrrrr}
\toprule
Test source & Error & $0\%$ & $3\%$ & $5\%$ & $10\%$\\
\midrule
Rounded Z & $E_a$ & 0.283 & 2.978 & 4.096 & 5.977\\
 & $E_\lambda$ & 0.063 & 1.304 & 2.132 & 4.441\\
\addlinespace[2pt]
Open helix & $E_a$ & 0.399 & 1.660 & 1.875 & 7.633\\
 & $E_\lambda$ & 0.039 & 1.381 & 2.298 & 4.723\\
\addlinespace[2pt]
Spatial Lissajous & $E_a$ & 0.608 & 1.170 & 2.478 & 6.511\\
 & $E_\lambda$ & 0.151 & 1.447 & 2.302 & 4.610\\
\addlinespace[2pt]
Spatial trefoil & $E_a$ & 0.790 & 1.965 & 3.629 & 5.064\\
 & $E_\lambda$ & 0.100 & 1.455 & 2.388 & 4.881\\
\bottomrule
\end{tabular}
\end{table}

\medskip\noindent
\textbf{Spectral reconstruction in a general known inhomogeneous medium.}
We next consider $n=1+q$, where
\begin{equation}
 q(x)=0.50B_{0.50}(x-(0.85,0.15,0.10))
      +0.35B_{0.40}(x-(-0.75,-0.25,0.15)),
 \label{eq:numerical-medium}
\end{equation}
and $B_R(x)=\exp(1-(1-|x|^2/R^2)^{-1})$ for $|x|<R$, with $B_R=0$ otherwise. The supports are disjoint and $1\leq n\leq1.5$. Candidate trajectories lie in $[-0.22,0.22]^3$, outside $\supp q$. The outgoing Green function is computed from the Lippmann--Schwinger equation~\cite{CK2019}, using different volume grids for data generation and inversion.

We fit 32 frequencies $k=0.4+0.4j$, $j=0,\ldots,31$, and use eight selection frequencies $k=0.6+1.6j$ and eight final-check frequencies $k=1.4+1.6j$, $j=0,\ldots,7$. The data are generated directly from the complete receiver spectra in \eqref{eq:spectral-forward}. Final-check data do not enter parameter selection. At each receiver, independent complex Gaussian noise is normalized separately on the 32 fitting frequencies and the original 16 held-out frequencies. Each source uses one fixed noise direction across levels. Relative noise is exact on these two blocks and their combination; the levels on the eight-frequency subsets may differ.

The reconstruction uses 16 trajectory and 29 intensity cubic controls, with low-frequency initialization and continuation as in Section~\ref{sec:batch}. The two regularization parameters are selected independently using the noise-dependent fitting-residual threshold and selection data, with an intensity regularization floor for noiseless fits. These heuristic rules were refined on the fixed noise realizations reported here. The candidate ranges, refinement rules, local linearized GCV criterion, selected parameters, and fitting and final-check residuals are given in the accompanying supplementary material.

We use a planar ellipse, an open helix, a planar figure-eight curve, and a spatial Lissajous curve, with the trajectories and intensities specified below Figures~\ref{fig:ellipse-reconstruction}--\ref{fig:spatial-reconstruction}. The formulas are used only for data generation and error evaluation. The figure-eight and spatial Lissajous paths intersect themselves at $s=1.5$ and $s=3.5$.

\par\smallskip\noindent
\begin{minipage}{\linewidth}
\centering
\includegraphics[width=\linewidth]{v5_reconstruction_header.png}\par
\par\noindent
\begin{minipage}{\linewidth}
\makeatletter
\def\@captype{figure}
\makeatother
\centering
\setlength{\abovecaptionskip}{3pt}
\setlength{\belowcaptionskip}{0pt}
\includegraphics[width=\linewidth]{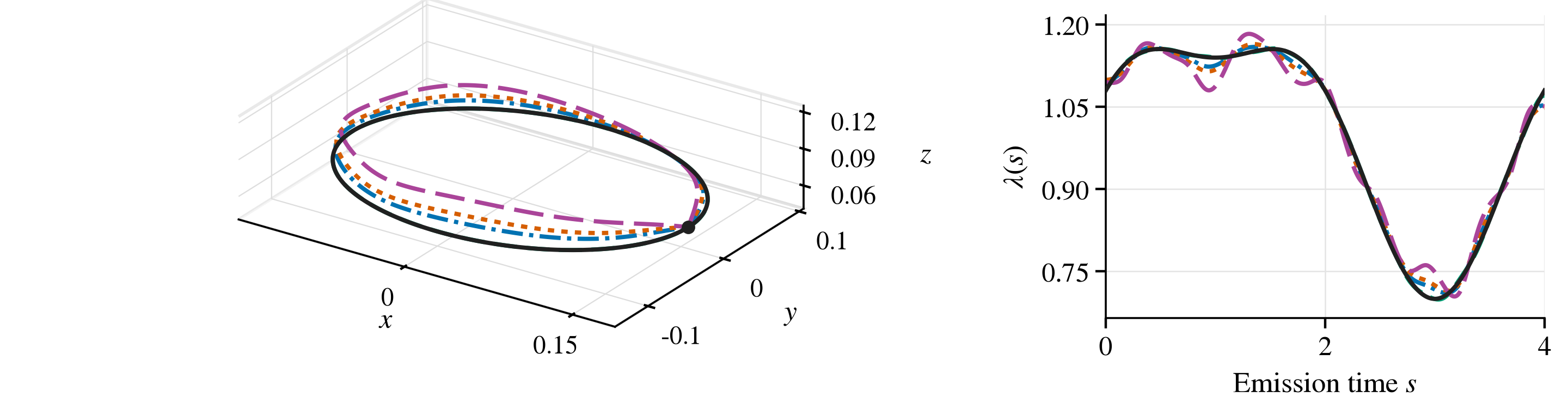}
\caption{Spectral reconstruction of the ellipse source.}
\label{fig:noisy-reconstruction}\label{fig:ellipse-reconstruction}
\vspace{1pt}
\noindent\hspace*{0.045\linewidth}%
\begin{minipage}[t][22pt][t]{0.559524\linewidth}
\centering\scriptsize
$\begin{aligned}[t]a_1(s)&=(0.02+0.15\cos(2\pi u),-0.02+0.11\sin(2\pi u),0.07),\end{aligned}$\par
\end{minipage}%
\hspace*{0.100714\linewidth}%
\begin{minipage}[t][22pt][t]{0.279762\linewidth}
\centering\scriptsize
\makebox[\linewidth][r]{$\begin{aligned}[t]\lambda_1(s)&=1+0.22\sin(2\pi u)+0.08\cos(4\pi u).\end{aligned}$}\par
\end{minipage}
\end{minipage}
\par\vspace{1pt}
\par\noindent
\begin{minipage}{\linewidth}
\makeatletter
\def\@captype{figure}
\makeatother
\centering
\setlength{\abovecaptionskip}{3pt}
\setlength{\belowcaptionskip}{0pt}
\includegraphics[width=\linewidth]{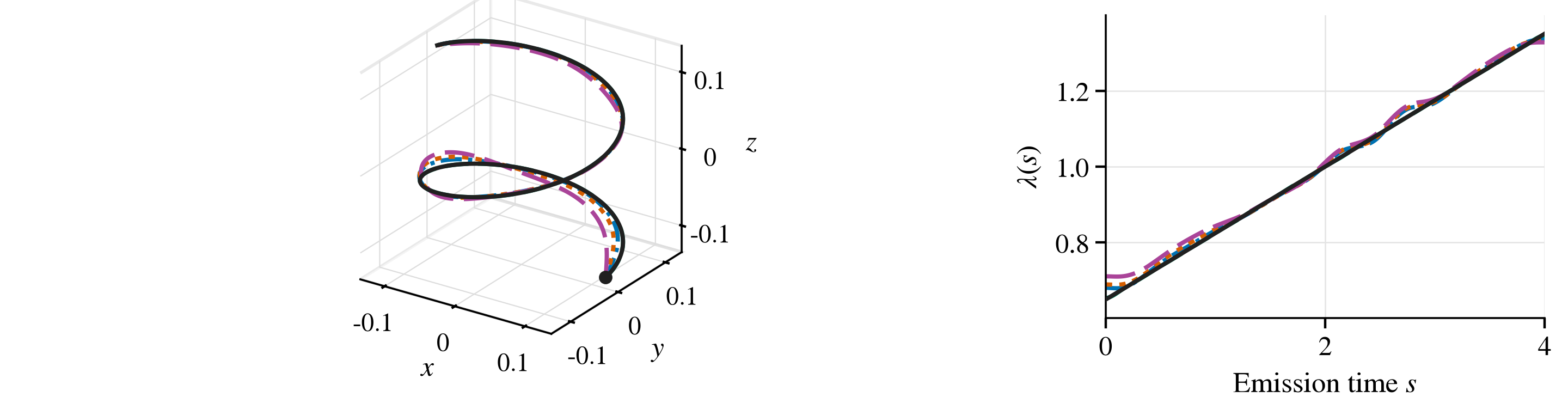}
\caption{Spectral reconstruction of the open-helix source.}
\label{fig:helix-reconstruction}
\vspace{1pt}
\noindent\hspace*{0.045\linewidth}%
\begin{minipage}[t][22pt][t]{0.559524\linewidth}
\centering\scriptsize
$\begin{aligned}[t]a_2(s)&=(0.12\cos(3\pi u),0.12\sin(3\pi u),-0.12+0.24u),\end{aligned}$\par
\end{minipage}%
\hspace*{0.100714\linewidth}%
\begin{minipage}[t][22pt][t]{0.279762\linewidth}
\centering\scriptsize
\makebox[\linewidth][r]{$\begin{aligned}[t]\lambda_2(s)&=0.65+0.70u.\end{aligned}$}\par
\end{minipage}
\end{minipage}
\par\vspace{1pt}
\par\noindent
\begin{minipage}{\linewidth}
\makeatletter
\def\@captype{figure}
\makeatother
\centering
\setlength{\abovecaptionskip}{3pt}
\setlength{\belowcaptionskip}{0pt}
\includegraphics[width=\linewidth]{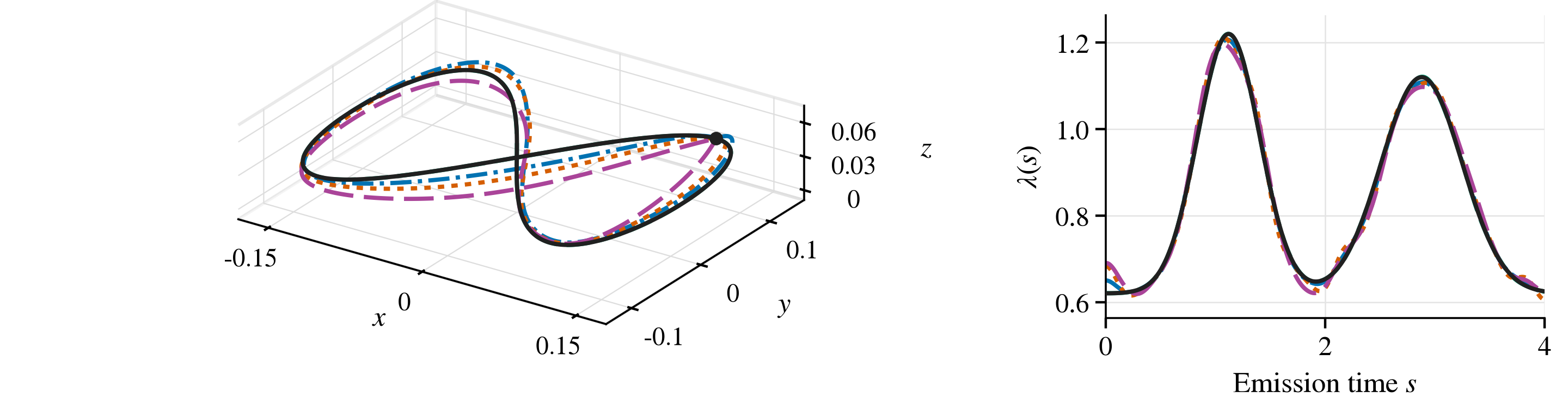}
\caption{Spectral reconstruction of the figure-eight source.}
\label{fig:figure-eight-reconstruction}
\vspace{1pt}
\noindent\hspace*{0.045\linewidth}%
\begin{minipage}[t][22pt][t]{0.559524\linewidth}
\centering\scriptsize
$\begin{aligned}[t]a_3(s)&=(0.16\sin\vartheta,0.12\sin(2\vartheta),0.04),\end{aligned}$\par
\end{minipage}%
\hspace*{0.100714\linewidth}%
\begin{minipage}[t][22pt][t]{0.279762\linewidth}
\centering\scriptsize
\makebox[\linewidth][r]{$\begin{aligned}[t]\lambda_3(s)&=0.62+0.60e^{-((u-0.28)/0.10)^2}\\[-1pt]&\quad+0.50e^{-((u-0.72)/0.13)^2}.\end{aligned}$}\par
\end{minipage}
\end{minipage}
\par\vspace{1pt}
\par\noindent
\begin{minipage}{\linewidth}
\makeatletter
\def\@captype{figure}
\makeatother
\centering
\setlength{\abovecaptionskip}{3pt}
\setlength{\belowcaptionskip}{0pt}
\includegraphics[width=\linewidth]{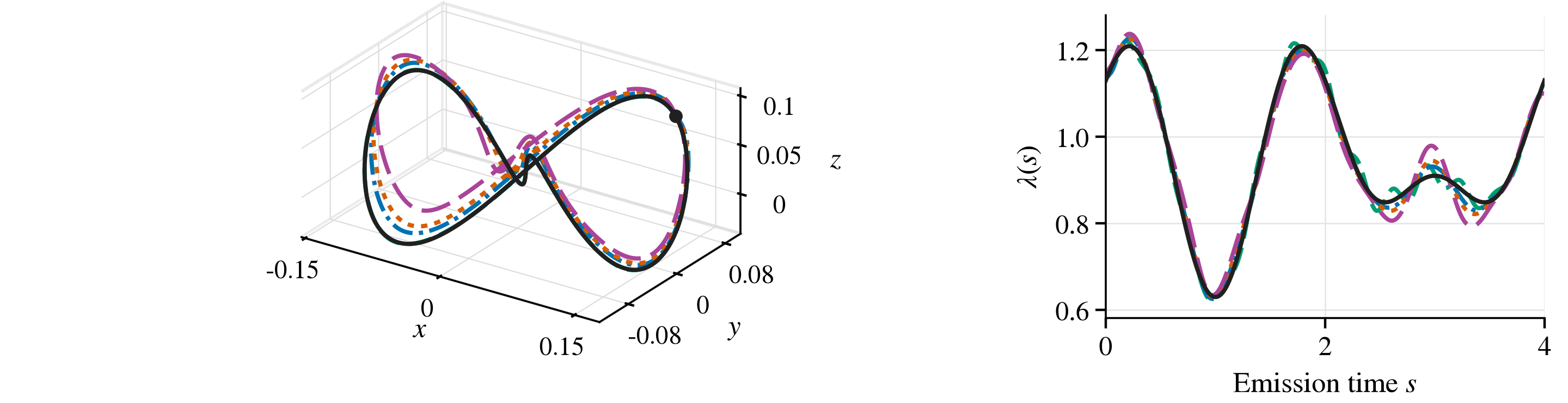}
\caption{Spectral reconstruction of the spatial Lissajous source.}
\label{fig:spatial-reconstruction}
\vspace{1pt}
\noindent\hspace*{0.045\linewidth}%
\begin{minipage}[t][22pt][t]{0.559524\linewidth}
\centering\scriptsize
$\begin{aligned}[t]a_4(s)&=(0.02+0.14\sin\vartheta,-0.015+0.10\sin(2\vartheta),0.03+0.06\sin(3\vartheta)),\end{aligned}$\par
\end{minipage}%
\hspace*{0.100714\linewidth}%
\begin{minipage}[t][22pt][t]{0.279762\linewidth}
\centering\scriptsize
\makebox[\linewidth][r]{$\begin{aligned}[t]\lambda_4(s)&=0.95+0.18\cos(4\pi u)+0.14\sin(6\pi u).\end{aligned}$}\par
\end{minipage}
\end{minipage}
\par\vspace{1pt}
\end{minipage}
\par\smallskip

Table~\ref{tab:reconstruction} lists the relative errors evaluated at 2001 equispaced emission times.
All 16 fits meet the stopping criteria and residual threshold. At $10\%$ noise, the relative trajectory errors range from $11.78\%$ to $19.83\%$, while the intensity errors range from $1.78\%$ to $3.09\%$. Compared with the coupled-parameter baseline on the same observations, independent parameter selection reduces trajectory errors in ten of the twelve noisy cases and intensity errors in all twelve; the ellipse has slightly larger trajectory errors at $3\%$ and $5\%$ noise.

\begin{table}[!htbp]
 \centering\small
\caption{Relative four-receiver reconstruction errors in percent, evaluated at 2001 equispaced emission times. Column headings give exact relative noise in the fitting data.}
 \label{tab:reconstruction}
 \begin{tabular}{lcrrrr}
 \toprule
 Test source & Error & $0\%$ & $3\%$ & $5\%$ & $10\%$\\
 \midrule
 Ellipse & $E_a$ & 0.032 & 4.801 & 7.513 & 13.161\\
  & $E_\lambda$ & 0.132 & 0.852 & 1.249 & 2.520\\
 \addlinespace[2pt]
 Open helix & $E_a$ & 0.046 & 3.951 & 6.685 & 13.177\\
  & $E_\lambda$ & 0.052 & 0.807 & 1.065 & 1.778\\
 \addlinespace[2pt]
 Figure-eight & $E_a$ & 0.069 & 6.717 & 7.410 & 11.778\\
  & $E_\lambda$ & 0.082 & 0.877 & 1.905 & 2.329\\
 \addlinespace[2pt]
 Spatial Lissajous & $E_a$ & 0.268 & 6.679 & 10.530 & 19.834\\
  & $E_\lambda$ & 1.406 & 1.186 & 1.673 & 3.087\\
 \bottomrule
 \end{tabular}
\end{table}

The numerical results demonstrate the effectiveness of both reconstruction methods for the simultaneous recovery of source trajectories and intensities from noisy measurements at four fixed receivers in known inhomogeneous media.

\appendix
\makeatletter
\renewcommand{\@seccntformat}[1]{\ifcsname appendix@#1@prefix\endcsname\csname appendix@#1@prefix\endcsname\fi\csname the#1\endcsname\quad}
\newcommand{\appendix@section@prefix}{\appendixname~}
\makeatother

\section{Well-posedness}\label{sec:direct}
\subsection{Causal solutions with point sources}

Write $H^s=H^s(\R^3)$ for the usual $L^2$-based Sobolev spaces.
Every finite Radon measure $\mu$ satisfies
\begin{equation}
\|\mu\|_{H^{-2}}\leq C\|\mu\|_{\M},
\label{eq:measure-sobolev}
\end{equation}
where $\|\mu\|_{\M}$ denotes the total variation norm.
Indeed, $|\widehat\mu(\xi)|\leq C\|\mu\|_{\M}$ and $\int_{\R^3}(1+|\xi|^2)^{-2}\dd\xi<\infty$. In particular, $\|\delta_y\|_{H^{-2}}$ is independent of $y$.
Under \eqref{eq:admissible}, the moving point source therefore belongs to $L^1(\R;H^{-2})$ and acts on test functions by
\begin{equation}
\langle F,\varphi\rangle
=\int_0^{T_0}\lambda(s)\varphi(a(s),s)\dd s.
\label{eq:source-pairing}
\end{equation}
The same conclusion holds for finite sums of such sources.

\begin{proposition}\label{prop:direct}
Under \eqref{eq:medium}, every $F\in L^1_{\loc}(\R;H^{-2})$ supported in $t\geq0$ has a unique causal distributional solution of \eqref{eq:wave}, satisfying
\begin{equation}
u\in C(\R;H^{-1}),\qquad
\partial_tu\in C(\R;H^{-2}).
\label{eq:negative-energy}
\end{equation}
For a compactly supported source, the solution has finite propagation speed with respect to $g$.
\end{proposition}

\begin{proof}
The operator $A=-n^{-1}\Delta$ is nonnegative and self-adjoint on $L^2(\R^3,n\dd x)$, with domain $H^2$.
It is associated with the closed quadratic form $\int_{\R^3}|\nabla v|^2\dd x$ on $H^1$~\cite[Chapter~VI]{Kato1995}.
Elliptic regularity, interpolation, and duality~\cite[Chapters~4--6]{Taylor2011} identify the scale generated by $I+A$ with the usual Sobolev spaces at the orders used below, with equivalent norms.
Multiplication by $n^{\pm1}$ is bounded on these spaces.

For $t\geq0$, the Duhamel formula is
\begin{equation}
u(t)=\int_0^t
\frac{\sin((t-s)\sqrt A)}{\sqrt A}\,n^{-1}F(s)\dd s.
\label{eq:duhamel}
\end{equation}
At spectral value zero, the quotient is defined by continuity as $t-s$.
The spectral theorem~\cite[Appendix~A]{Taylor2011} gives a uniformly bounded map $H^{-2}\to H^{-1}$ for $0\leq t-s\leq L$, since
\[
\sup_{\sigma\geq0}
(1+\sigma)^{1/2}
\left|\frac{\sin(\theta\sqrt\sigma)}{\sqrt\sigma}\right|
\leq C_L,\qquad 0\leq\theta\leq L.
\]
The time derivative is given by the cosine propagator acting on $H^{-2}$.
Approximation of $F$ in $L^1(0,L;H^{-2})$ proves \eqref{eq:negative-energy}, the equation, and the zero initial values.
Uniqueness among causal distributional solutions follows by duality with smooth backward solutions.

The local energy estimate~\cite[Chapter~6]{Taylor2011} and approximation give finite propagation for distributional data.
The propagation metric of $\partial_t^2-n^{-1}\Delta$ is $g=n\dd x^2$, so a source at $(y,s)$ can influence $(x,t)$ only if $t\geq s+d_g(x,y)$.
\end{proof}

\subsection{Distributional pressure traces}

For sources with fixed compact spatial support $K\Subset\Omega$, choose a collar $N$ of $\Gamma$ such that $\overline N\cap K=\varnothing$.
For $\eta\in C_c^\infty(\R)$, set
\begin{equation}
U_\eta=\int_\R u(t)\eta(t)\dd t\in H^{-1}(\R^3).
\label{eq:time-pairing}
\end{equation}
Since $F=0$ on $N\times\R$, the wave equation gives
\begin{equation}
\Delta U_\eta=nU_{\eta''}\quad\text{on }N.
\label{eq:paired-elliptic}
\end{equation}
All derivatives of $\eta$ have the same compact time support.
Applying interior elliptic regularity~\cite[Chapters~4 and 5]{Taylor2011} to \eqref{eq:paired-elliptic} for $\eta$ and its derivatives gives successively $H^1_{\loc}(N)$, $H^3_{\loc}(N)$, and higher regularity.
Sobolev embedding then yields $U_\eta\in C^\infty(N)$.
On each fixed compact time support, the interior estimates bound the smooth seminorms of $U_\eta$ near $\Gamma$ by finitely many test-function seminorms of $\eta$.

For $\psi\in C^\infty(\Gamma)$, define
\begin{equation}
\langle u|_{\Gamma\times\R},\psi(x)\eta(t)\rangle
:=\int_\Gamma\psi(x)U_\eta(x)\dd S_x.
\label{eq:weak-trace}
\end{equation}
The preceding bounds and the Schwartz kernel theorem~\cite[Chapter~5]{Hormander2003} extend this pairing uniquely to a distribution on $\Gamma\times\R$.
By \eqref{eq:time-pairing}, temporal convolution commutes with this trace:
\begin{equation}
(\rho_\eps*_t u)|_\Gamma
=\rho_\eps*_t(u|_\Gamma).
\label{eq:trace-commute}
\end{equation}
For sources supported in $D_0\times[0,T_0]$, finite propagation gives $u|_\Gamma=0$ for $t<d_g(D_0,\Gamma)$, where $d_g(D_0,\Gamma)>0$.
Thus equality of two traces on $\Gamma\times(0,T)$ implies equality on $\Gamma\times(-\infty,T)$.

\section{Green function representation}\label{app:Green}

\begin{lemma}[Source-parametrized Hadamard construction]
\label{lem:parametric-Hadamard}
Fix $\widehat H>0$, and let $Y$ be a bounded open neighborhood of the compact set $D$.
Let $g=n_\eps\dd x^2$ and $\mathcal Q=\partial_t^2+L$, with $L=-\Delta_g+V$, be as in \eqref{eq:conjugation}.
Suppose that, for every $y$ near $\overline Y$, the exponential map $\exp_y$ maps $\{|v|_{g(y)}<2\widehat H\}$ diffeomorphically onto $B_g(y,2\widehat H)$.
Then the coefficients in \eqref{eq:finite-parametrix} are jointly smooth on 
\[
\mathcal N=\{(x,y):y\in Y,\ \varrho=d_g(x,y)<2\widehat H\},
\]
 independent of $t$, and satisfy $U_0>0$.
For each $r<2\widehat H$, their derivatives of any fixed finite order are uniformly bounded for $y\in\overline Y$ and $\varrho\leq r$.
Let $F_N$ be the regular residual multiplied by the fixed spatial and temporal cutoffs in Step~3 below.
For each integer $m\geq0$, $F_N$ is jointly $C^m$ in $(x,y,t)$ if $N\geq m+2$.
For $m\geq4$, the causal solution of $\mathcal Q W=-F_N$ is jointly $C^1$ for $0\leq t\leq\widehat H$, with bounds uniform in $y\in\overline Y$.
\end{lemma}

\begin{proof}
The inverse function theorem applied to $(y,v)\mapsto(y,\exp_yv)$ gives a jointly smooth inverse $v=v(x,y)$.
Thus $\gamma_s(x,y)=\exp_y(sv(x,y))$ and $\varrho^2=|v(x,y)|_{g(y)}^2$ are jointly smooth, including at $x=y$.
Define the positive density $\mu$ by $(\exp_y)^*\dd V_g=\mu(y,v)\dd V_{g(y)}(v)$, where $\dd V_{g(y)}$ is the metric volume on the tangent space.
The product spacetime metric has the same normal-coordinate density, independent of $t$.
The Hadamard transport formulas~\cite[Proposition~2.3.1]{BGP2007} therefore give
\[
\begin{split}
U_0(x,y)&=\mu(y,v(x,y))^{-1/2},\\
U_j(x,y)&=-jU_0(x,y)\int_0^1 s^{j-1}
U_0(\gamma_s(x,y),y)^{-1}
(L_xU_{j-1})(\gamma_s(x,y),y)\dd s.
\end{split}
\]
Differentiation under the integral proves joint smoothness and time independence, and compactness gives the uniform bounds.

The regular residual is a smooth coefficient times $\one_{\{t\geq\varrho\}}(t^2-\varrho^2)^{N-1}$.
In the smooth coordinates $w=n_\eps(y)^{1/2}v(x,y)$, we have $\varrho=|w|$.
Writing $p=N-1$, the cone factor is $C^{p-1}$ away from the vertex.
Its interior derivatives of order $k\leq m$ are $O((|t|+|w|)^{2p-k})$ and extend continuously by zero at the vertex when $p\geq m+1$.
This proves joint $C^m$ regularity for $N\geq m+2$.
The cutoffs give $F_N$ a common compact spatial support for $y\in\overline Y$.

Since $\mathcal Q$ acts only in $(x,t)$, difference quotients in $y$ and the energy estimate~\cite[Chapter~6]{Taylor2011} yield
\[
\begin{aligned}
&\sup_{\substack{y\in\overline Y\\0\leq t\leq\widehat H}}
\sum_{|\alpha|\leq1}
\left(\|\partial_y^\alpha W(t,y)\|_{H_x^3}
+\|\partial_t\partial_y^\alpha W(t,y)\|_{H_x^2}\right)\leq C\sup_{y\in\overline Y}\sum_{|\alpha|\leq1}
\|\partial_y^\alpha F_N\|_{L^1(0,\widehat H;H_x^2)}.
\end{aligned}
\]
For $m\geq4$, the forcing depends $C^1$ on $y$ in this norm, and the right-hand side is uniformly finite.
The embeddings $H^3(\R^3)\hookrightarrow C^1$ and $H^2(\R^3)\hookrightarrow C^0$ then give joint $C^1$ regularity of $W$.
\end{proof}

\begin{proposition}\label{prop:small-medium}
Let $D\subset\R^3$ be compact, let $x_1,\ldots,x_4\notin D$ be distinct receivers, and fix a real $q_0\in C_c^\infty(\Omega)$.
Suppose that $H>\max_{i,y\in D}|x_i-y|$.
For all sufficiently small $|\eps|$, the medium $n_\eps=1+\eps q_0$ satisfies \eqref{eq:Green-hypothesis} for $0<t<H$, with smooth $\tau_i,A_i$, positive $A_i$, and tails admitting $C^1$ extensions near $D\times[0,H]$.
If the Euclidean travel-time matrix has a uniform positive smallest singular value on $D$, then \eqref{eq:travel-rank} also holds for sufficiently small $|\eps|$.
\end{proposition}

\subsection{Proof of Proposition~\ref{prop:small-medium}}

\begin{proof}
Choose $\widehat H>H$ and a bounded open neighborhood $Y$ of $D$ such that $\overline Y$ is disjoint from the receivers and $\max_{i,y\in\overline Y}|x_i-y|<H$.
Write $g_\eps=n_\eps\dd x^2$.
For sufficiently small $|\eps|$, we have $1/2\leq n_\eps\leq3/2$, and comparison with Euclidean line segments gives $\max_{i,y\in\overline Y}d_{g_\eps}(x_i,y)<H$. The larger time interval will provide the required extensions near $t=H$.

\medskip\noindent\textit{Step 1: source-centered normal balls.}
The geodesic equation and its linearization give
\begin{equation}
\sup_{\substack{y\in\R^3\\|v|\leq4\widehat H}}
\|D_v\exp_y^{g_\eps}(v)-I\|
\leq C_{\widehat H}|\eps|,
\label{eq:small-medium-exp}
\end{equation}
where the constant depends on $\widehat H$ and finitely many bounds on $q_0$.
Indeed, writing the geodesic equation on $0\leq r\leq1$ as $\gamma''=-\mathsf C(\gamma)(\gamma',\gamma')$, we have $\|\mathsf C\|_{C^1}\leq C|\eps|$ and $|\gamma'|\leq\sqrt3|v|$ by energy conservation.
For a variation $J=D_v\gamma\,e$,
\[
J''=-D\mathsf C(\gamma)(J)(\gamma',\gamma')
-2\mathsf C(\gamma)(J',\gamma'),
\qquad J(0)=0,\quad J'(0)=e.
\]
Its integral equation and Gronwall's inequality give \eqref{eq:small-medium-exp}.

Choose $|\eps|$ so that $C_{\widehat H}|\eps|<1/2$.
Integration along a line segment in the initial velocity gives
\[
|\exp_y(v)-\exp_y(w)-(v-w)|\leq\tfrac12|v-w|,
\qquad |v|,|w|\leq4\widehat H.
\]
Hence $\exp_y$ is both injective and nonsingular on this ball.
Uniform equivalence to the Euclidean metric makes $g_\eps$ complete, so minimizing geodesics exist by the Hopf--Rinow theorem~\cite[Chapter~6]{Lee2018}.
Every geodesic of length less than $2\widehat H$, parametrized on $[0,1]$, has initial Euclidean speed less than $2\sqrt2\widehat H<4\widehat H$.
Such a geodesic must minimize, since a shorter minimizing geodesic to the same endpoint would contradict injectivity.
Consequently, $\exp_y$ maps $\{|v|_{g_\eps(y)}<2\widehat H\}$ diffeomorphically onto $B_{g_\eps}(y,2\widehat H)$, as required in Lemma~\ref{lem:parametric-Hadamard}.

\medskip\noindent\textit{Step 2: a finite parametrix.}
By \eqref{eq:conjugation}, setting $\varpi=n_\eps^{1/4}$ reduces $P$ to $\mathcal Q=\partial_t^2-\Delta_{g_\eps}+V$.
Let $K_{\mathcal Q}$ be the causal kernel of $\mathcal Q$, normalized with respect to $\dd V_{g_\eps}$.
On the normal domain above, put $\varrho=d_{g_\eps}(x,y)$.
The finite Hadamard--Riesz construction~\cite[Sections~1.4 and 2.1--2.3]{BGP2007} gives
\begin{equation}
\begin{split}
K_N(x,y,t)
={}&\frac{U_0(x,y)}{4\pi\varrho}\delta(t-\varrho)+\one_{\{t\geq\varrho\}}
\sum_{\ell=1}^N c_\ell U_\ell(x,y)(t^2-\varrho^2)^{\ell-1},
\end{split}
\label{eq:finite-parametrix}
\end{equation}
where $c_\ell$ are the constants in the Riesz normalization, and the expression at the vertex is understood distributionally.
The residual has the form
\[
\mathcal R_N:=\mathcal Q K_N-\delta(t)\delta_y^{g_\eps}
=b_N(x,y)\one_{\{t\geq\varrho\}}(t^2-\varrho^2)^{N-1},
\]
with smooth $b_N$.
Fix $m\geq4$ and $N\geq m+2$ so that Lemma~\ref{lem:parametric-Hadamard} applies.

\medskip\noindent\textit{Step 3: causal correction and normalization.}
Choose a fixed smooth cutoff $\chi(\varrho^2)$ equal to one for $\varrho\leq3\widehat H/2$ and supported in $\varrho<7\widehat H/4$.
Since $K_N$ is supported in $t\geq\varrho$, the commutator error from this cutoff is supported in $t\geq3\widehat H/2$.
Choose $\xi\in C_c^\infty(\R)$ equal to one on $[0,\widehat H]$ and supported below $3\widehat H/2$, and set $F_N=\chi(\varrho^2)\xi(t)\mathcal R_N$. Extending by zero outside the normal domain gives a jointly $C^m$ forcing with common compact support in $(x,t)$.
For $t<\widehat H$, we therefore have $\mathcal Q(\chi K_N)=\delta(t)\delta_y^{g_\eps}+F_N$.

Let $W$ solve $\mathcal Q W=-F_N$ causally.
Lemma~\ref{lem:parametric-Hadamard} gives its joint $C^1$ regularity, uniformly for $y\in\overline Y$.
The support of $F_N$, finite propagation, and the triangle inequality imply 
\[
\supp W(\cdot,y,\cdot)
\subset\{(x,t):t\geq d_{g_\eps}(x,y)\}.
\]
 Causal uniqueness now yields $K_{\mathcal Q}=\chi K_N+W$ for $t<\widehat H$.

Since $\dd V_{g_\eps}=n_\eps^{3/2}\dd x$, the Green function normalized with respect to the Euclidean source measure is
\begin{equation}
G_P(x,y,t)=\varpi(x)\varpi(y)K_{\mathcal Q}(x,y,t).
\label{eq:kernel-normalization}
\end{equation}
At each receiver, $\tau_i(y):=d_{g_\eps}(x_i,y)<H$, so the spatial cutoff equals one and $A_i(y)=\frac{\varpi(x_i)\varpi(y)U_0(x_i,y)}{4\pi\tau_i(y)}>0$. Because $W(x_i,y,t)=0$ for $t<\tau_i(y)$, it can be included in the Heaviside tail in \eqref{eq:Green-hypothesis}.
The polynomial tail in \eqref{eq:finite-parametrix} is smooth near the separated receivers, while $W$ is jointly $C^1$.
Together with $\widehat H>H$, these facts give the claimed $C^1$ extensions of $V_i$ near $D\times[0,H]$.
Compactness gives a positive lower bound on $A_i$ and the required norm bounds.

Finally, smooth dependence of geodesics on $\eps$ and the inverse function theorem give $\tau_{i,\eps}\to|x_i-\cdot|$ in $C^2(D)$.
Continuity of the smallest singular value then preserves the uniform positive rank bound for sufficiently small $|\eps|$.
\end{proof}

\section{Discrete reconstruction error}\label{app:marching}

\subsection{Interior sampling and interpolation}\label{app:interior-sampling}

For each receiver, retain all samples strictly inside $I_i$ when its endpoints are known exactly.
When the endpoints are computed from approximate travel times, retain all samples inside
\begin{equation}
\widehat I_i^{\,\rm safe}
=\big(\widehat\tau_i(a_-)+\epsilon_K,\,
T_0+\widehat\tau_i(a_+)-\epsilon_K\big)\subset I_i.
\label{eq:safe-crop}
\end{equation}
The inclusion follows from \eqref{eq:kernel-error}.
Require at least two retained samples for each receiver.
The bound $\beta_i\geq\theta$ gives $|I_i|\geq\theta T_0$ and $|\widehat I_i^{\,\rm safe}|\geq\theta T_0-4\epsilon_K$.
Thus $4\epsilon_K+2\Delta<\theta T_0$ guarantees at least two retained samples.
Let $J_i$ be their piecewise linear interpolant with constant extensions, and let $W_i$ be the true waveform extended constantly using its interior endpoint limits.

\begin{lemma}\label{lem:interpolation}
Under these sampling and cropping rules,
\begin{equation}
\|J_i-W_i\|_{L^\infty(\R)}
\leq\eta:=\delta+L_w(\Delta+2\epsilon_K).
\label{eq:interpolation-error}
\end{equation}
The $2\epsilon_K$ term is omitted when the arrival interval endpoints are known exactly.
\end{lemma}

\begin{proof}
Cropping removes at most $2\epsilon_K$ from each end of $I_i$, and the next strictly interior grid point is at distance at most $\Delta$.
Hence the first and last retained samples are within $\Delta+2\epsilon_K$ of the corresponding endpoints of $I_i$.
Since all samples inside the chosen interval are retained, consecutive retained sample times remain separated by at most $\Delta$.
The nonnegative interpolation weights give an error of at most $\delta+L_w\Delta$ between retained samples.
Outside their range, the constant extensions give an error bounded by $\delta$ plus $L_w$ times the corresponding endpoint gap.
This proves \eqref{eq:interpolation-error}, with the cropping term omitted when the endpoints are known exactly.
\end{proof}

\subsection{Proof of Theorem~\ref{thm:march-error}}\label{app:guard-thresholds}

For the exact source, set 
\[
r_i(s)=s+\tau_i(a(s)),\qquad
p_i(s)=\frac{\lambda(s)A_i(a(s))}{\beta_i(s)},
\]
 and define $M_*(s)$ by its rows 
\[
(M_*(s))_{i,:}
=\big(p_i(s)\nabla\tau_i(a(s))^T,-A_i(a(s))\big).
\]
 The constants used in the checks are
\begin{equation}
\begin{gathered}
B=1+M_v\max_i\|\nabla\tau_i\|_{L^\infty(K)},\qquad
p_*=\lambda_*A_*/B,\\
K_M=\frac{B}{A_*\kappa}
\left(1+\frac{1+M_v}{\lambda_*}\right),\qquad
m_*=K_M^{-1}.
\end{gathered}
\label{eq:guard-constants}
\end{equation}
By the source bounds and \eqref{eq:matrix-factor}, using the Euclidean operator norm, 
\[
p_i(s)\geq p_*,\qquad
\|M_*(s)^{-1}\|\leq K_M,\qquad
\sigma_{\min}(M_*(s))\geq m_*.
\]

\begin{proof}
Use the Lipschitz representatives of $a'$ and $\lambda$ on $[0,T_0]$, and set 
\[
a_n^*=a(s_n),\qquad v_n^*=a'(s_n),\qquad
\ell_n^*=\lambda(s_n),
\]
  
\[
e_n=|a_n-a_n^*|,\qquad
d_n=|v_n-v_n^*|+|\ell_n-\ell_n^*|.
\]
 The constant $C$ depends only on the quantities listed in Theorem~\ref{thm:march-error} and may change from line to line.

\medskip\noindent\textit{Step 1: tail and waveform errors.}
Suppose that the preceding steps have passed all checks and that the current evaluation points lie in their prescribed domains.
The source bounds and the $C^1$ bounds on $V_i$ make the exact tail integrand uniformly Lipschitz in emission time, so the left-endpoint quadrature error is at most $Ch$.
Moreover,
\[
\begin{aligned}
|\widehat r_i^n-r_i(s_n)|&\leq\epsilon_K+Ce_n,\\
|J_i(\widehat r_i^n)-W_i(r_i(s_n))|
&\leq\eta+C(\epsilon_K+e_n).
\end{aligned}
\]
At the exact history points, $r_i(s_n)-s_j\geq\tau_i(a(s_j))$, so $V_i^e$ agrees with $V_i$.
Comparing the history sums using \eqref{eq:kernel-error} and the Lipschitz continuity of $V_i^e$ gives
\begin{equation}
|p^n-p_*^n|
\leq C\big(\eta+\epsilon_K+h+e_n
+h\sum_{j=0}^{n-1}(e_j+d_j)\big),
\label{eq:discrete-p-error}
\end{equation}
where $p_*^n=(p_1(s_n),\ldots,p_4(s_n))^T$.
This also holds at $n=0$ because $W_i$ uses the interior waveform limits.

\medskip\noindent\textit{Step 2: discrete error estimates.}
Set $M_*^n=M_*(s_n)$ and
\[
z_n=\binom{v_n}{\ell_n},\qquad
z_n^*=\binom{v_n^*}{\ell_n^*}.
\]
The coefficient bounds give $\|M_n-M_*^n\| \leq C\big(|p^n-p_*^n|+e_n+\epsilon_K\big)$. Once the checks before solving pass, $\|M_n^{-1}\|\leq2/m_*$, and the exact identity $M_*^nz_n^*=-p_*^n$ yields
\[
z_n-z_n^*
=M_n^{-1}\big(-(p^n-p_*^n)
-(M_n-M_*^n)z_n^*+e_n^{\rm lin}\big).
\]
Consequently,
\begin{equation}
d_n\leq C\big(q_{\rm err}+e_n
+h\sum_{j=0}^{n-1}(e_j+d_j)\big).
\label{eq:discrete-d-error}
\end{equation}
This estimate does not require the checks after the current solve.
The Euler update and Taylor's estimate give
\begin{equation}
e_{n+1}\leq e_n+hd_n+Ch^2,\qquad
e_n\leq h\sum_{j=0}^{n-1}d_j+Ch,\qquad e_0=0.
\label{eq:discrete-e-error}
\end{equation}
Substitution into \eqref{eq:discrete-d-error} and interchange of the finite sums give 
\[
d_n\leq Cq_{\rm err}+Ch\sum_{j=0}^{n-1}d_j.
\]
 The discrete Gronwall inequality bounds $d_n$ and the position errors, including the next node, by $Cq_{\rm err}$ on all accepted steps.

\medskip\noindent\textit{Step 3: verification of the checks.}
Suppose that all preceding steps have passed and consider step $n$.
Then $a_n$ and all earlier history points lie in $K$.
Writing $\tau_{\max}=\max_{i,y\in K}\tau_i(y)$, we have 
\[
\tau_{\min}-\epsilon_K
\leq\widehat r_i^n-s_j
\leq T_0+\tau_{\max}+\epsilon_K,
\qquad j<n.
\]
 Thus sufficiently small $\varepsilon_0$ keeps all current tail evaluations in $K\times[0,H]$.
The preceding error bounds now imply $|p^n-p_*^n|+\|M_n-M_*^n\|\leq Cq_{\rm err}$. Choosing $\varepsilon_0$ so that $C\varepsilon_0<\min\{p_*,m_*\}/2$ ensures both checks before solving.
Equation~\eqref{eq:discrete-d-error} then gives $d_n\leq Cq_{\rm err}$, which, together with the source bounds and \eqref{eq:kernel-error}, ensures the checks in \eqref{eq:post-guards} after a further reduction of $\varepsilon_0$.
Also, \eqref{eq:discrete-e-error} gives $e_{n+1}\leq Cq_{\rm err}<\rho$, so $a_{n+1}\in K$.
Starting from $a_0=a_-$, induction therefore proves that all steps are completed and 
\[
\max_{0\leq n\leq N}e_n
+\max_{0\leq n<N}d_n\leq Cq_{\rm err}.
\]

\medskip\noindent\textit{Step 4: function errors and arrival-time monotonicity.}
The Lipschitz continuity of $a'$ and $\lambda$, together with the piecewise definitions of $a_h$ and $\lambda_h$, gives
\[
\|a_h-a\|_{W^{1,\infty}(0,T_0)}
+\|\lambda_h-\lambda\|_{L^\infty(0,T_0)}\leq C\left(
\max_{0\leq n\leq N}e_n
+\max_{0\leq n<N}d_n+h\right)
\leq Cq_{\rm err}.
\]
This proves \eqref{eq:march-error}.

Convexity of $K$ ensures that $a_h(s)\in K$ throughout the interval.
By the chain rule, for almost every $s\in(0,T_0)$,
\[
\frac{\dd}{\dd s}\big(s+\widehat\tau_i(a_h(s))\big)
=1+\nabla\widehat\tau_i(a_h(s))\cdot a_h'(s).
\]
The approximation bound and the Lipschitz continuity of $\nabla\tau_i$ imply
\[
\left|\nabla\widehat\tau_i(a_h)\cdot a_h'
-\nabla\tau_i(a)\cdot a'\right|\leq C\left(\epsilon_K
+\|a_h-a\|_{L^\infty(0,T_0)}
+\|a_h'-a'\|_{L^\infty(0,T_0)}\right)
\leq Cq_{\rm err}.
\]
Since $\beta_i\geq\theta$, reducing $\varepsilon_0$ so that $C\varepsilon_0\leq\theta/2$ gives
\[
\frac{\dd}{\dd s}\big(s+\widehat\tau_i(a_h(s))\big)
\geq\theta-Cq_{\rm err}\geq\theta/2
\quad\text{a.e. in }(0,T_0).
\]
This proves the arrival-time monotonicity for every receiver.
\end{proof}

\section*{Acknowledgements}
 X. Xu is partially supported by the National Key Research and Development Program of China (No. 2024YFA1012303), National Natural Science Foundation of China (No. 12525112), and the Open Research Project of Innovation Center of Yangtze River Delta, Zhejiang University. Y. Zhao is supported by National Natural Science Foundation of China (No. 12371421).

\section*{Conflict of Interest Statement}

The authors declare that they have no conflict of interest.

\section*{Data Availability Statement}
The numerical experiments use synthetic data generated using the models and parameters described in this article.
The data supporting the numerical results are available from the corresponding author upon reasonable request.

\end{document}